\documentclass[leqno,11pt,a4paper,twoside]{article}
\usepackage{authblk}

\usepackage{amsfonts}
\usepackage{amsmath}
\usepackage{graphicx,graphics,color,rotating}
\usepackage{ulem}
\DeclareGraphicsExtensions{.pdf, .jpg, .tif, .png}
\usepackage{pgfplots}

\usepackage{tikz}
\usetikzlibrary{shapes,arrows,shadows,patterns,pgfplots.groupplots,decorations.markings,chains,calc}

\newtheorem{theorem}{Theorem}

\newtheorem{definition}[theorem]{Definition}

\newtheorem{lemma}[theorem]{Lemma}

\newtheorem{proposition}[theorem]{Proposition}
\newtheorem{remark}[theorem]{Remark}

\newenvironment{proof}[1][Proof]{\noindent\textbf{#1.} }{\ \rule{0.5em}{0.5em}}
\providecommand{\abs}[1]{\left\lvert#1\right\rvert}
\providecommand{\pr}[1]{\left(#1\right)} %(.)
\providecommand{\pp}[1]{\left[#1\right]} %[.]
\providecommand{\set}[1]{\left\lbrace#1\right\rbrace} %{.}
\title{Stochastic porous media equation with Robin boundary conditions, gravity-driven infiltration and multiplicative noise}

\author{
 {Ioana CIOTIR$^{1}$, Dan GOREAC$^{2,3,4}$, Juan LI$^{2,5}$, Antoine TONNOIR$^{1}$}\\
 {$^1$\small Normandie University, INSA de Rouen Normandie, LMI (EA 3226 – FR CNRS 3335), \\ \small 76000 Rouen, France}\\
{$^2$\small School of Mathematics and Statistics, Shandong University, Weihai,\\ \small Weihai 264209, P.~R.~China.}\\
{$^3$\small \'{E}cole d'actuariat, Universit\'{e} Laval, QC, Qu\'{e}bec, Canada.}\\
{$^4$\small LAMA, Univ Gustave Eiffel, UPEM, Univ Paris Est Creteil, CNRS,\\ \small F-77447 Marne-la-Vall\'{e}e, France.}\\
{$^5$\small Research Center for Mathematics and Interdisciplinary Sciences, Shandong University, \\ \small Qingdao 266237, P. R. China.}\\
{\small \it E-mails: ioana.ciotir@insa-rouen.fr,\,\ dan.goreac@univ-eiffel.fr,\\ \small \it juanli@sdu.edu.cn,\,\ antoine.tonnoir@insa-rouen.fr }
}

\begin{document}
\maketitle

\begin{abstract}
We aim at studying a novel mathematical model associated to a physical phenomenon of infiltration in an homogeneous porous medium. The particularities of our system are connected to the presence of a gravitational acceleration term proportional to the level of saturation, and of a Brownian multiplicative perturbation. Furthermore, the boundary conditions intervene in a Robin manner with the distinction of the behavior along the inflow and outflow respectively. We provide qualitative results of well-posedness, the investigation being conducted through a functional approach.

\end{abstract}


\noindent{MSC2020: 
35A01, %Existence problems for PDEs: global existence, local existence, non-existence
60H15, %Stochastic partial differential equations
76S99, %Flows in porous media
47H05. %Monotone operators and generalizations

}

\section{Introduction}

\bigskip

Our contribution primarily focuses on the qualitative study of a fairly involved mathematical model which describes the \emph{infiltration} of a liquid (water) in
an homogeneous porous medium (soil) taking into account the influence of the
\emph{gravitational acceleration} proportional to the level of saturation. In order to provide the readers with a keener insight, we choose to present the physical model prior to further considerations on the different aspects of the state-of-the-art.

\bigskip

\subsection{The physical model}

\bigskip
Along with delineating the physical model, we choose to present the units of measure for the different notions in paranthesis, such as to guarantee coherence.\\
First we consider an
incompressible fluid with constant density denoted by $\rho \left( \frac{kg}{m^{3}}%
\right) .$ Then, we consider a reference elementary volume $%
V_{r}\left( m^{3}\right) $ belonging to the flow domain, and we distinguish between $V_{v}$\ the
\emph{volume of voids} and $V_{w}$ the \emph{volume of water} in $V_{r}$. The flow is said
to be \emph{unsaturated} as long as all pores are not filled with water.

\noindent We further introduce the notion of soil \emph{moisture} as a scalar dimensionless notion by setting 
\begin{equation*}
\theta =\frac{V_{w}}{V_{r}}.
\end{equation*}

\noindent As in classical fluid dynamics theory, we combine the Darcy law and
the equation of continuity (or mass conservation) in order to get a
Richards' type equation which takes into consideration the influence of the
gravitational acceleration. The originality of the model consists in the
fact that we assumed that the influence of the gravity on the dynamic of the
flow is proportional with the moisture. In turn, this translates into further mathematical technical difficulties. The relevance of such models is illustrated shortly after this brief description of the model of interest.

In a classical manner, we invoke \emph{Darcy's law} written down as
\begin{equation*}
q=-\frac{k}{\mu }\left( \nabla p-\rho \theta gi_{3}\right),
\end{equation*}%
where

\begin{itemize}
\item $q$ is the \emph{fluid volumetric flux} ($m/s$),

\item $k$ is the \emph{isotropic permeability} ($m^{2}$),

\item $\mu $ is the coefficient of \emph{viscosity} ($kg/\left( m\cdot s\right) $),

\item $p$ is the \emph{pressure} ($kg/\left( m\cdot s^{2}\right) $),

\item $\rho $ is the \emph{density} ($kg/m^{3}$),

\item $\theta $ is the \emph{moisture} (a dimensionless physical constant)

\item $gi_{3}$ is the \emph{gravity} acceleration vector ($g$ being expressed in $m/s^{2}$).
\end{itemize}

On the other hand, the equation of mass conservation yields
\begin{equation*}
\frac{\partial \left( \rho \theta \right) }{\partial t}+\textit{div}\left(
\rho q\right) =f,
\end{equation*}%
where $\rho ,\theta $ and $q$ are as above and $f$ is a water \emph{source}.

\noindent By combining the two equations and keeping in mind that $\rho $\ is constant
since the fluid is incompressible, it follows that%
\begin{equation*}
\frac{\partial \theta }{\partial t}-\frac{k}{\mu }\textit{div}\left( \nabla
p-\rho \theta gi_{3}\right) =f.
\end{equation*}
The attentive reader will have noticed that we have used $f$ instead of a $\rho$-renormalized source. We have chosen to do so in order to keep notations to a manageable minimum.
The hydraulic theory ensures that the water capacity is nothing but the
derivative of the moisture with respect to the pressure, i.e.,%
\begin{equation*}
C\left( p\right) =\frac{d\theta }{dp}.
\end{equation*}%
As a consequence, one can assume that a primitive of $C$ denoted by $C^{\ast }$ is a
single-valued positive, twice differentiable, strongly monotonically
increasing and concave function such that $\theta =C^{\ast }\left( p\right) $. Since it is more convenient to work with the variable $\theta $, we
introduce the inverse $\left( C^{\ast }\right) ^{-1}$, in order to obtain an
equation of the form%
\begin{equation*}
\frac{\partial \theta }{\partial t}-\Delta \frac{k}{\mu }\left( C^{\ast
}\right) ^{-1}\left( \theta \right) +\frac{k\rho g}{\mu }\textit{div}\left(
\theta i_{3}\right) =f.
\end{equation*}

\noindent By denoting $\Psi \left( r\right) =\frac{k}{\mu }\left( C^{\ast }\right)
^{-1}\left( r\right) $ and $K=\dfrac{k\rho g}{\mu }$, one can rewrite the
previous equation in the slightly more general form%
\begin{equation}\label{eq0}
\frac{\partial \theta }{\partial t}-\Delta \Psi \left( \theta \right) +K%
\textit{div}\left( \theta i_{3}\right) =f.
\end{equation}

\bigskip 

\subsubsection{Numerical illustrations}

\bigskip 

Let us present some numerical illustrations of the impact of the gravity term in the model. In Figure \ref{fig1}, we consider the porous medium mathematically modeled through $\Psi(r) = r^3$, and we can observe that, for the same source on the top of the domain, the solution is quite different. The humidity is more attracted towards bottom with the gravity term. This type of result motivates the study of the model with the gravity term.

\begin{figure}[h]
	\begin{center}
		\begin{tikzpicture}
			\begin{scope}[shift={(0,0)},scale=1.0,transform shape]
				\node at (0,0) {\includegraphics[trim={20cm 0cm 20cm 0cm},clip, width=0.24\textwidth]{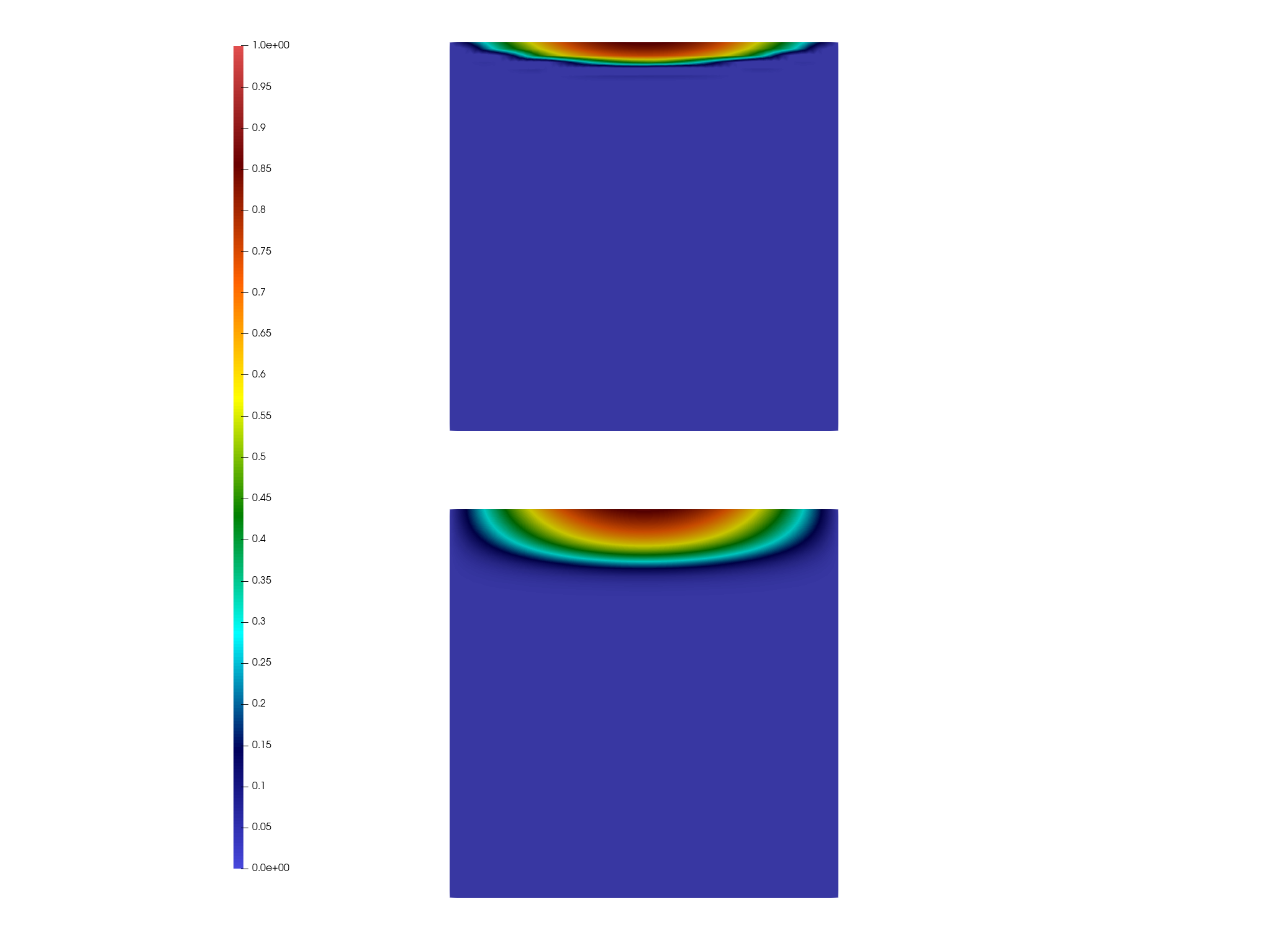}};
			\end{scope}
			\begin{scope}[shift={(3,0)},scale=1.0,transform shape]
				\node at (0,0) {\includegraphics[trim={20cm 0cm 20cm 0cm},clip, width=0.24\textwidth]{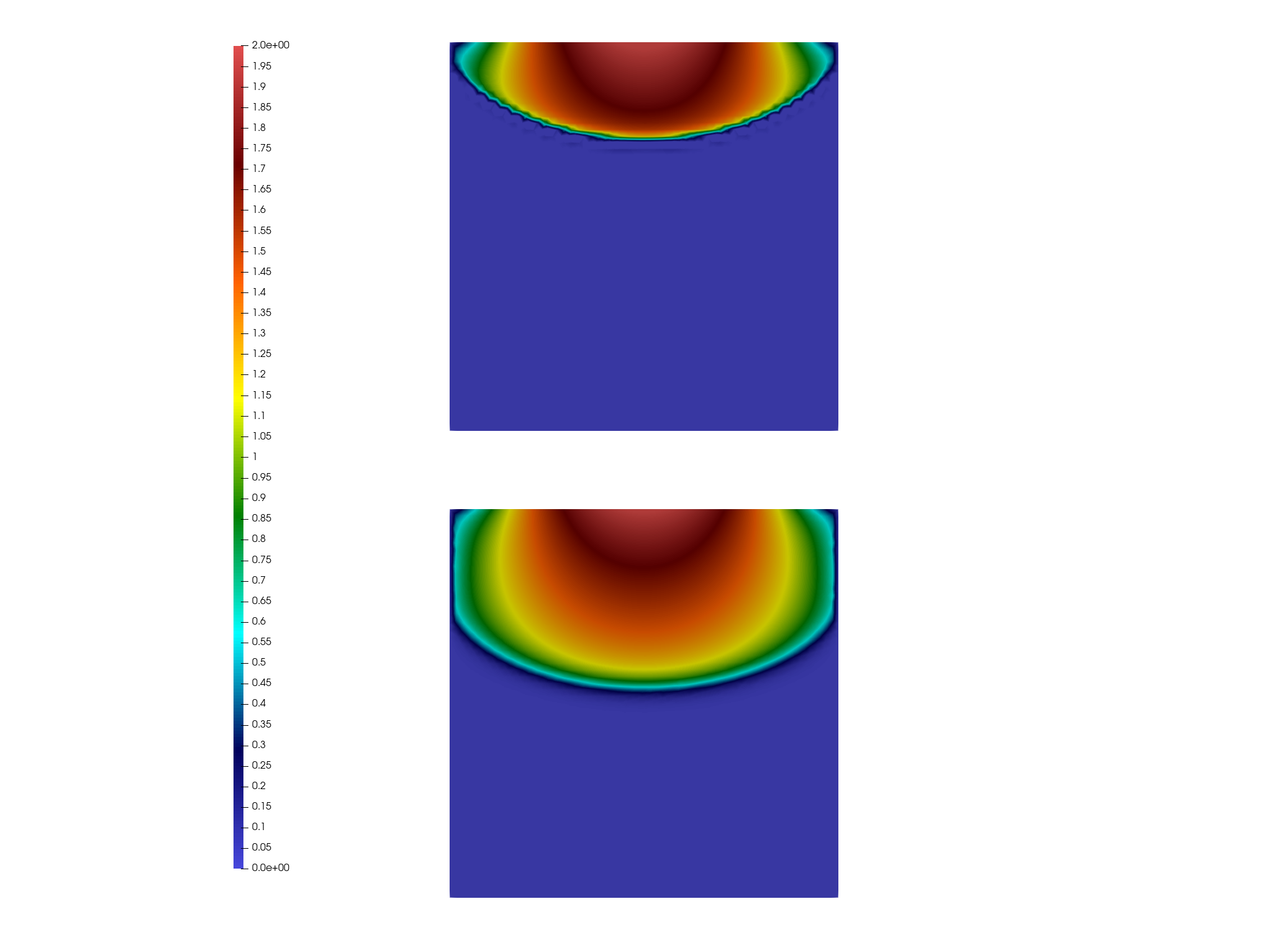}};
			\end{scope}
			\begin{scope}[shift={(6,0)},scale=1.0,transform shape]
				\node at (0,0) {\includegraphics[trim={20cm 0cm 20cm 0cm},clip, width=0.24\textwidth]{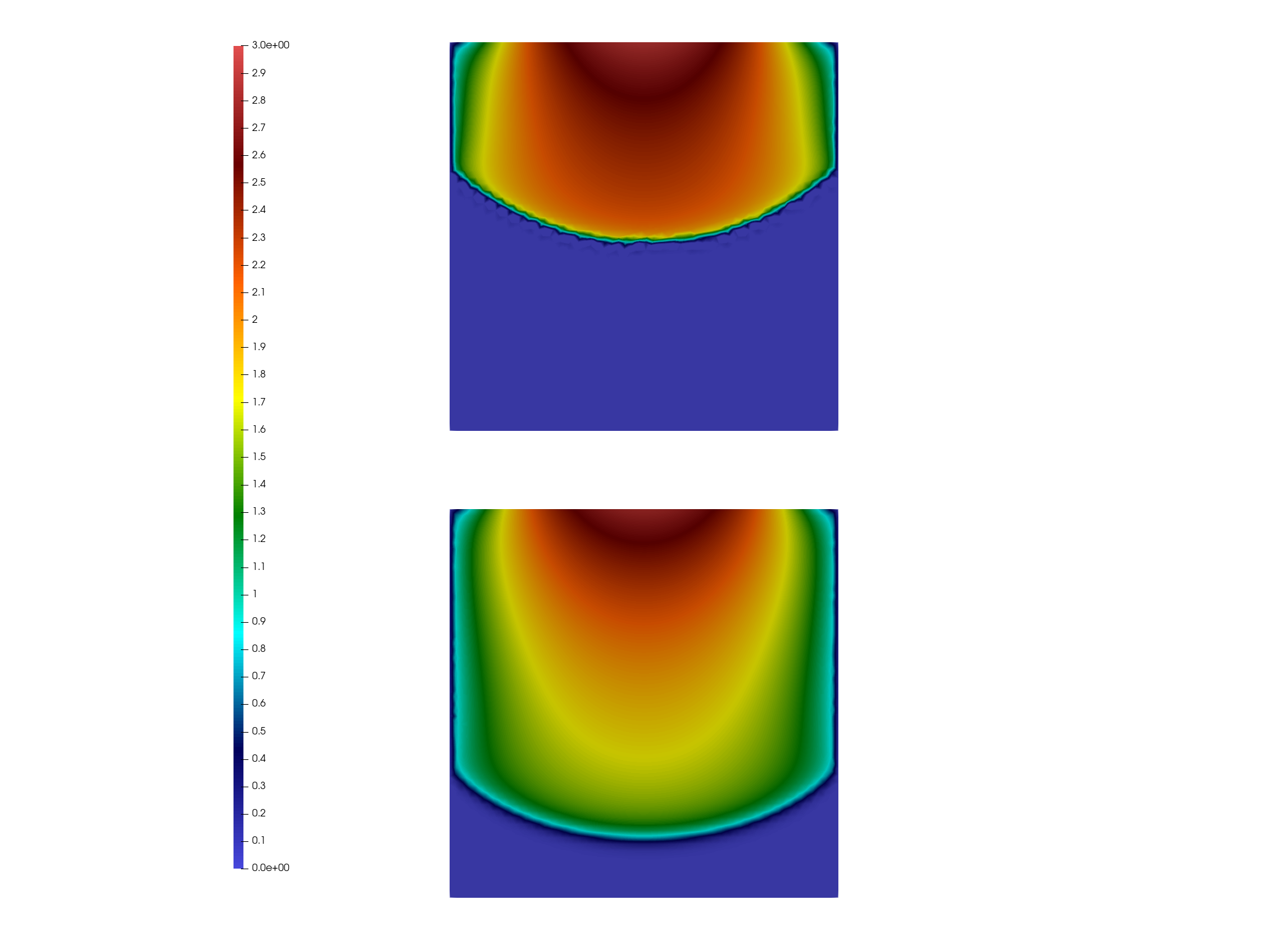}};
			\end{scope}
			\begin{scope}[shift={(9,0)},scale=1.0,transform shape]
				\node at (0,0) {\includegraphics[trim={20cm 0cm 20cm 0cm},clip, width=0.24\textwidth]{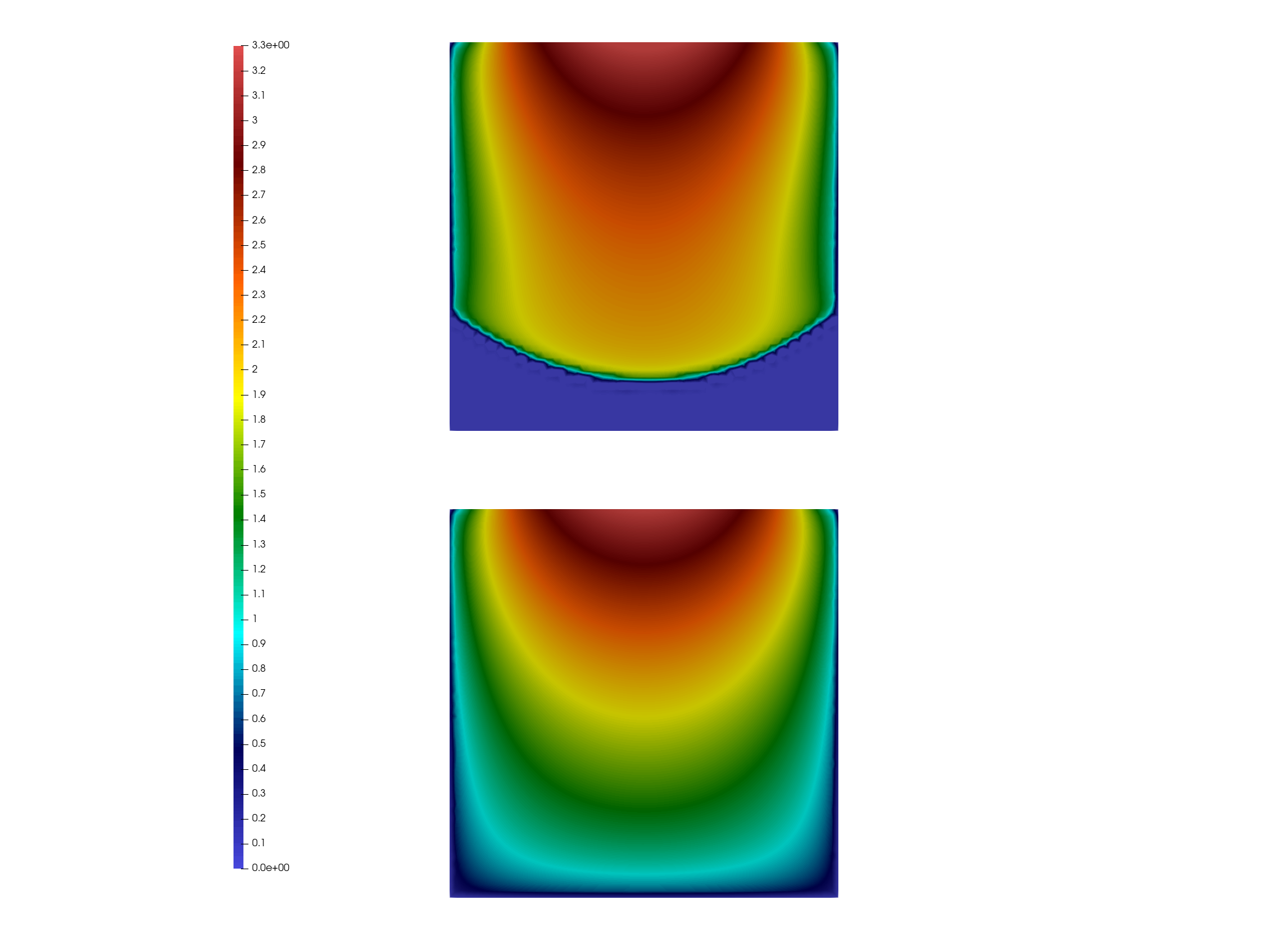}};
			\end{scope}
			\draw[->,thick] (-1,-3.2) -- (10,-3.2);
			%\node at (-3.7,-0.15) {$h$};
			%\draw[<->,thick] (3.3,-1.4) -- (3.3,-0.1);
			\node[rotate=90] at (10.8,1.8) {$g\cdot i_3  = 0$};
			\node[rotate=90] at (10.8,-1.5) {$g\cdot i_3  < 0$};
			\node at (4.8,-3.7) {Time};
% 			\node at (-4.5,-1.0) {$\Gamma_{in}$};
% 			\node at (-0.2,1.1) {$\Gamma$};
% 			\node at (0.2,-1.1) {$\Gamma$};
		\end{tikzpicture}
	\end{center}
	\caption{Illustration of the diffusion of the humidity inside the ground during time, from left to right. On the top, without gravity, on the bottom with the gravity term.}\label{fig1}
\end{figure}
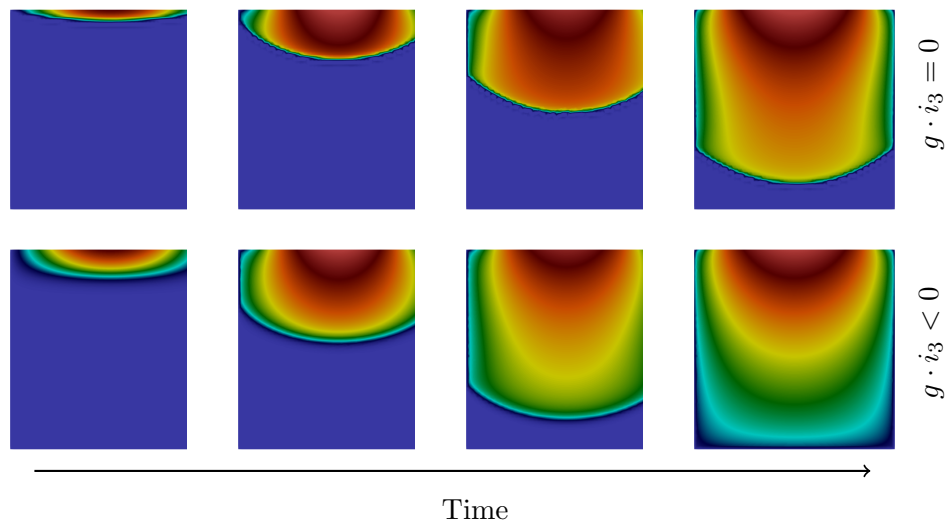	
\bigskip \bigskip 

\subsection{A mathematical model}

\bigskip

We are now able to introduce the rigorous mathematical model in the spirit of \eqref{eq0}, in which the source $f$ is random and pertaining to a Brownian noise.

\noindent To this purpose, we consider a domain $\mathcal{O}$\ which is an open bounded subset of 
$\mathbb{R}^{3}$ with smooth boundary $\partial \mathcal{O}=\Gamma $. This boundary is splited into two parts denoted respectively by $\Gamma _{s}$ (on the surface) and $\Gamma _{u}$\ (underground), and such that 
$\Gamma =\Gamma _{s}\cup \overline{\Gamma _{u}}$ and $\Gamma _{s}\cap \Gamma
_{u}=\emptyset $. This domain extends from the soil surface $\Gamma _{s}$ to
an underground boundary $\Gamma _{u}$\ which is supposed to have a variable
permeability due to the presence of different types of soils. More
precisely, we are interested in the infiltration of the rain from surface boundary of the soil (source which is assumed to be known
as the average value of the previous precipitations) and underground
water, taking also into
account the influence of the gravity on the diffusion process and a linear
multiplicative stochastic noise coming from the errors of the measurement.

\noindent Our mathematical model is the following 
\begin{equation}
\left\{ 
\begin{array}{ll}
dX-\Delta \Psi \left( X\right) dt+K\dfrac{\partial X}{\partial x_{3}}%
dt=\sigma \left( X\right) dW_{t}, & \quad \textnormal{ on }\Omega \times \left( 0,T\right)
\times \mathcal{O}, \\ 
X\left( 0\right) =x, & \quad \textnormal{ on }\mathcal{O}, \\ 
\left( KXi_{3}-\nabla \Psi \left( X\right) \right) \nu =s, & \quad \textnormal{ on }\Omega
\times \left( 0,T\right) \times \Gamma _{s} , \\ 
\left( KXi_{3}-\nabla \Psi \left( X\right) \right) \nu -\alpha \Psi \left(
X\right) =u, & \quad \textnormal{ on }\Omega \times \left( 0,T\right) \times \Gamma _{u}.%
\end{array}%
\right.  \label{ecu}
\end{equation}

\noindent The solution $X\left( \omega ,t,\xi \right) $ of the equation \eqref{ecu}
describes the moisture on a path $\omega \in\Omega$. The solution is sought with respect to a complete
filtered probability space $\left( \Omega ,\mathcal{F},\left( \mathcal{F}%
_{t}\right) _{t\geq 0},\mathbb{P}\right) $, at each moment $t>0$ and at each point $%
\xi $ in space.\\

\noindent The nonlinear operator $\Psi $ contains the physical pieces of information related to
the hydraulic conductivity of the soil, while the term $K\dfrac{\partial }{%
\partial x_{3}}$ concerns the influence of the gravity on the diffusion
process. We denote by $\nu $ the outward normal to $\Gamma $ and by $i_{3}$ the unit
vector along the axis $Ox_{3}$ directed downwards.\\

\noindent The functions $s$ and $u$ are supposed to be known on $\Sigma _{s}=\left(
0,T\right) \times \Gamma _{s}$ and $\Sigma _{u}=\left( 0,T\right) \times
\Gamma _{u}$ respectively. Note that the equation on $\Gamma _{s}$\ expresses the continuity of the
normal component of the inflow flux and the equation on $\Gamma _{u}$
describes the behavior of the outflow. Finally, the function $\alpha $\ gives the variable permeability of the soil on the
underground boundary $\Gamma _{u}$.

\bigskip 

\subsection{State of the art and main contributions}

\bigskip 

\noindent Several different models connected to the porous media
diffusion are available. For general results concerning the existence and uniqueness of
the solution in the deterministic case for different models including slow, fast
and super fast diffusions in saturated or unsaturated flows, the reader is kindly referred to, e.g., \cite%
{Aaron}, \cite{Vasquez} and \cite{marinoschi}.

The stochastic porous media equation has been intensively studied during the recent
years. The reader is invited to consuls \cite{PM} for a collection of results
concerning the existence and uniqueness of the solution for the case with
Dirichlet homogeneous boundary conditions. For some critical cases, the papers \cite%
{eu-superfast}, \cite{eu-unsaturated}, \cite{IDR} are a good start. For some further qualitative properties of
the solutions, see \cite{finite-ext}, \cite{IDI1}, \cite{IDI2}, \cite{Jonas}, 
\cite{gess}, and for homogenization results see \cite{Trotter} and \cite%
{eu-conv}.

To the best of our knowledge, the only type of boundary conditions which was considered previously in the literature for the stochastic porous media equations are of Dirichlet homogeneous nature. On the other hand, the influence of the gravity on the diffusion process was never taken into consideration for the stochastic porous media equations so far, again, to our best knowledge.

In the present work, we develop a new and more realistic physical model that takes into account the influence of the gravity,  as well as the different flows on the boundary, through Robin type boundary conditions. 

Our model implies several technical difficulties at the mathematical level. Since the usual framework in $L^1$ or $H^{-1}$ are not appropriate for the Robin-type boundary conditions, one has to consider a different space whose norm accounts for the different components of the flow, as well as the source on the boundary. For this reason, several results, especially related to the construction and the properties of the stochastic noise, and also to the activity of the porous media operator needed to be adapted to the present case.
The estimates and the convergences of the proof are of a rather involved technical nature, because of the presence of several extra terms topping the more complicated spaces and norms previously mentioned.

\bigskip

The paper is organized as follows. In section \ref{Sec2} we explain the fundamental mathematical setting and the standing assumptions employed throughout the paper. The Gelfand triple with the appropriate norms is introduced in Subsection \ref{Subsec2.1}. The standing assumptions are gathered in Subsection \ref{Subsec2.2}; they concern the driving porous coefficient, the boundary inputs and the construction of the noise coefficient. The notion of abstract solution is specified in Subsection \ref{Subsec2.3}, while the remaining Subsection \ref{Subsec2.4} concerns the (quasi-) accretivity properties of the abstract nonlinear operator.
The main result makes the object of Section \ref{Sec3}, through Theorem \ref{ThMain}. Its proof is divided into several steps, and consists in two approximations, indexed by $\lambda, \varepsilon>0$. Showing that the approximating $\lambda, \varepsilon$-problem is well-posed in $L^2$ relies on Lemma \ref{lema2}, whose proof is relegated to the Appendix. Passing to the limit as $\varepsilon\rightarrow 0$ relies on the estimates in Lemma \ref{lema3}, whose proof is given at the end of Section \ref{Sec3}. The remaining arguments are gathered in the proof of the main result.

\section{Mathematical tools and assumptions}\label{Sec2}
\subsection{A Functional framework}\label{Subsec2.1}
In order to properly define the model, we need to introduce a
functional setting which is appropriated to our problem. In particular, since we intend to rely on the Gelfand-triple arguments, we introduce the norms on both the primal and the dual space in order to render the arguments compatible with the homogeneous Robin conditions.\\

\noindent To this purpose, we denote by $V$ the space $H^{1}\left( \mathcal{O}\right) $ endowed with the following norm%
\begin{equation*}
\left\vert x\right\vert _{V}=\left( \int_{\mathcal{O}}\left\vert \nabla
x\right\vert ^{2}d\xi +\int_{\Gamma _{u}}\alpha \left( \xi \right)
\left\vert x\right\vert ^{2}d\xi \right) ^{1/2},
\end{equation*}%
where $\alpha $\ is a strictly positive, continuous function, upper bounded and lower bounded away from $0$.
One can easily check that the norm above is equivalent with the classical
Hilbert norm of $H^{1}\left( \mathcal{O}\right) $ (see Theorem 2.7 in the Appendix of \cite{marinoschi}).

\noindent We denote by $V^{\prime }$ the dual of the space $V$ which is equipped with
the scalar product 
\begin{equation*}
\left\langle x,\overline{x}\right\rangle _{V^{\prime }}=x\left( \varphi
\right) ,\quad \forall x,\overline{x}\in V^{\prime },
\end{equation*}%
where $\varphi \in V$ satisfies the boundary value problem%
\begin{equation*}
\left\{ 
\begin{array}{ll}
-\Delta \varphi =\overline{x}, & \quad \text{on }\mathcal{O} \\ 
\dfrac{\partial \varphi }{\partial \nu }+\alpha \varphi =0, & \quad \text{on 
}\Gamma _{u} \\ 
\dfrac{\partial \varphi }{\partial \nu }=0, & \quad \text{on }\Gamma _{s}%
\text{.}%
\end{array}%
\right.
\end{equation*}

Here, $x\left( \varphi \right) $\ represents the value of the functional $x$
computed at $\varphi \in V$, or the pairing between $V$ and $V^{\prime }$ which
reduces to the scalar product in $L^{2}\left( \mathcal{O}\right) $ in the spirit of the Gelfand triple $V\subset L^{2}\left( \mathcal{O}\right) \subset V^{\prime }$.\\

We further consider the Laplace operator with Robin boundary conditions, and,
more precisely we are interested in the eigenvalues/ eigenfunctions problem 
\begin{equation}
\left\{ 
\begin{array}{ll}
-\Delta e_{j}=\lambda _{j}e_{j}, & \quad \text{in }\mathcal{O} \\ 
\dfrac{\partial e_{j}}{\partial \nu }+\alpha e_{j}=0, & \quad \text{on }%
\Gamma _{u} \\ 
\dfrac{\partial e_{j}}{\partial \nu }=0, & \quad \text{on }\Gamma _{s}\text{.%
}%
\end{array}%
\right.  \label{neumann}
\end{equation}

By the same argument as in the proof of Theorem 6.2 from \cite{limita} we
obtain the existence of a complete orthonormal system $\left\{ e_{k}\right\}
_{k}$ in $L^{2}\left( \mathcal{O}\right) $ of eigenfunctions of the Laplace
operator with Robin boundary conditions. We denote by $\left\{ \lambda _{k}\right\} _{k}$ the corresponding
sequence of real and increasing eigenvalues. We denote by $C$ a general constant that may
change from one line to another. The fundamental boundedness properties are gathered in the first part of our Appendix.

\subsection{Assumptions}\label{Subsec2.2}
Throughout the paper and unless stated otherwise, we will enforce the following.
\begin{itemize}
\item[1)] The function $\Psi :\mathbb{R\rightarrow R}$ is assumed to be a
maximal monotone operator in $\mathbb{R}$, continuous, differentiable, increasing, i.e. 
\begin{equation*}
\left( \Psi \left( r\right) -\Psi \left( s\right) \right) \left( r-s\right)
\geq C_{0}\left\vert r-s\right\vert ^{2},\quad \forall r,s\in \mathbb{R},
\end{equation*}%
for some $C_0\geq 0$. Furthermore, we ask that
\begin{equation*}
\left\{ 
\begin{array}{l}
\left\vert \Psi \left( r\right) \right\vert \leq C_{1}\left\vert
r\right\vert ^{m}+C_{2},\quad \forall r\in \mathbb{R}, \\ 
j\left( r\right) :=\int_{0}^{r}\Psi \left( s\right) ds\geq C_{3}\left\vert
r\right\vert ^{m+1}+C_{4}r^{2} -C_5,\quad \forall r\in \mathbb{R}, \\ 
\Psi \left( 0\right) =0%
\end{array}%
\right. 
\end{equation*}%
where $C_{i}>0,~i\in\set{1,\ldots,5}$ and $1\leq m$. When $K\neq 0$, we ask that $\Psi$ be strictly increasing, i.e., $C_0>0$.
We note that since $\Psi $\
is increasing, the mean value theorem implies that 
\begin{equation*}
r\Psi \left( r\right) \geq j\left( r\right) ,\quad \forall r\in \mathbb{R}.
\end{equation*}

\item[2)] The function $\alpha :\Gamma _{u}\rightarrow \left[ \alpha
_{m},\alpha _{M}\right] $ is positive and continuous such that \[0<\alpha
_{m}<\alpha \left( \xi \right) <\alpha _{M},\] for all $\xi\in\Gamma_u$.

\item[3)] The functions $u$ and $s$ belong to $L^{2}\left( \Sigma
_{u}\right) $ and $L^{2}\left( \Sigma _{s}\right) $ respectively, and are\
the traces in each $t$ of some functions $\widetilde{u}$ and $\widetilde{s}$
which belong to $H^{1}\left( \mathcal{O}\right) $.

\item[4)] The Wiener process $W$ is assumed to be cylindrical on $%
L^{2}\left( \mathcal{O}\right) $ and given by 
\begin{equation*}
W\left( t\right) =\sum_{j=1}^{\infty }\beta _{j}\left( t\right) e_{j},
\end{equation*}%
where $\left\{ \beta _{j}\right\} $ is a sequence of mutually independent
Brownian motions on the filtered probability $\left( \Omega ,\mathcal{F}%
,\left( \mathcal{F}_{t}\right) _{t~},\mathbb{P}\right) $ and $\left\{
e_{j}\right\} $ is the orthonormal basis in $L^{2}\left( \mathcal{O}\right) $%
\ defined by the Robin-Laplace operator.

\item[5)] In the noise we assume that $\Sigma $\ is a linear operator from $%
V^{\prime }$ to the space of Hilbert-Schmidt operators $L_{2}\left(
L^{2}\left( \mathcal{O}\right) ;V^{\prime }\right) $ of the form%
\begin{equation*}
\Sigma \left( x\right) h=\sum_{j=1}^{\infty }\mu _{j}\left( h,e_{j}\right)
_{2}xe_{j},\quad \forall x\in V^{\prime },~h\in L^{2}\left( \mathcal{O}%
\right) .
\end{equation*}%
In order to guarantee the Hilbert-Schmidt condition, i.e. 
$
\sum_{j=1}^{\infty }\left\vert \Sigma \left( x\right) e_{j}\right\vert
_{V^{\prime }}^{2}<\infty ,\quad \forall x\in V^{\prime }$,
we need to enforce the convergence of the serial%
\begin{equation*}
\sum_{j=1}^{\infty }\mu _{j}^{2}\left\vert xe_{j}\right\vert _{V^{\prime
}}^{2}\leq C\sum_{j=1}^{\infty }\mu _{j}^{2} \pr{1+\lambda _{j}^{\frac{d+1}{2}}}\left\vert
x\right\vert _{V^{\prime }}^{2}.
\end{equation*}%
This inequality is guaranteed due to the estimates in the Appendix.
For these reasons, we ask that
\begin{equation*}
\sum_{j=1}^{\infty }\mu _{j}^{2} \pr{1+\lambda _{j}^{\frac{d+1}{2}}}\leq C.
\end{equation*}
\end{itemize}

\begin{remark}
\begin{enumerate}
\item Note that, from the last two assumptions, we get that the form of the noise coefficient is%
\begin{equation*}
\Sigma \left( X\right) dW_{t}=\sum_{j=1}^{\infty }Xe_{j}\mu _{j}d\beta
_{j}\left( t\right) ,\quad \forall X\in V^{\prime },~t\geq 0.
\end{equation*}
\item The assumption on $j$ in the first point guarantees, in particular, that the range of $\Psi$ covers the whole $\mathbb{R}$.
\item The constant $C_5$ can be chosen to be positive and this guarantees coverage of cases when $\Psi$ is null on a compact domain, as it is the case, for instance, in Stefan-like problems.
\end{enumerate}
\end{remark}
\subsection{A notion of solution }\label{Subsec2.3}
In order to explain the choice of the solution for equation (\ref{ecu}),\ we
shall first rewrite the equation as a problem set on the space $V^{\prime }$.%
\begin{equation}
\left\{ 
\begin{array}{ll}
dX\left( t\right) +A\left( X\left( t\right) \right) dt+F_{u}\left( t\right)
dt+F_{s}\left( t\right) dt=\Sigma \left( X\left( t\right) \right) dW_{t} 
,\quad \textnormal{for } t\in \left( 0,T\right), \\ 
X\left( 0\right) =x,
\end{array}%
\right.  \label{ecu2}
\end{equation}%
where the nonlinear operator $
A:D\left( A\right) \in V^{\prime }\rightarrow V^{\prime }$
is defined by 
\begin{equation*}
_{V^{\prime }}\left\langle A\left( X\right) ,\varphi \right\rangle
_{V}=\int_{\mathcal{O}}\left( \nabla \Psi \left( X\right) \cdot \nabla
\varphi -KX\frac{\partial \varphi }{\partial x_{3}}\right) d\xi
+\int_{\Gamma _{u}}\alpha Tr\left( \Psi \left( X\right) \right) Tr\left(
\varphi \right) d\sigma ,
\end{equation*}%
for all $\varphi \in V$. The natural domain is%
\begin{equation*}
D\left( A\right) =\left\{ X\in L^{2}\left( \mathcal{O}\right) ;~\Psi \left(
X\right) \in H^{1}\left( \mathcal{O}\right) \right\} .
\end{equation*}
The reader is invited to note that, since $\Psi \left( X\right) \in H^{1}\left( \mathcal{O}\right) $,
then the trace $Tr\left( \Psi \left( X\right) \right) \in L^{2}\left( \Gamma\right) $ and therefore the operator $A$ is well defined.\\

\noindent The second and the third operators from the equation (\ref{ecu2}) belong to $%
L^{2}\left( 0,T;V^{\prime }\right) $, and they are defined by 
\begin{equation*}
F_{u}\left( t\right) \left( \varphi \right) =-\int_{\Gamma _{u}}u\left(
t\right) Tr\left( \varphi \right) d\sigma ,\quad \forall \varphi \in V,
\end{equation*}%
respectively by%
\begin{equation*}
F_{s}\left( t\right) \left( \varphi \right) =-\int_{\Gamma_s}s\left(
t\right) Tr\left( \varphi \right) d\sigma ,\quad \forall \varphi \in V.
\end{equation*}

\noindent With the notations above, the equation (\ref{ecu2}) is well posed as a Cauchy
problem in $V^{\prime }$ and we use this formulation to construct a
variational solution by testing the equality against the eigenfunctions, i.e., by requiring
\begin{eqnarray*}
&&\left\langle X\left( t\right) ,e_{j}\right\rangle _{V^{\prime
}}+\int_{0}^{t}\left\langle A\left( X\left( r\right) \right)
,e_{j}\right\rangle _{V^{\prime }}dr +\int_{0}^{t}\left\langle F_{u}\left( r\right) ,e_{j}\right\rangle
_{V^{\prime }}dr+\int_{0}^{t}\left\langle F_{s}\left( r\right)
,e_{j}\right\rangle _{V^{\prime }}dr \\
&&=\left\langle x,e_{j}\right\rangle _{V^{\prime }}+\int_{0}^{t}\left\langle
\Sigma \left( X\left( r\right) \right) dW_{r},e_{j}\right\rangle _{V^{\prime
}},\quad \forall t\in \left[ 0,T\right] ,
\end{eqnarray*}%
$\mathbb{P}$-a.s., and for all $e_{j}$ from the orthonormal basis specified
above.\\
\noindent In order to define a more detailed form of the solution we make some explicit computations which are also going to be useful throughout the paper in order to pass from the nonlinear operators of type $A$ to the underlying real functions $\Psi$, respectively to $K$. We have:%
\begin{eqnarray}
\left\langle A\left( X\left( r\right) \right) ,e_{j}\right\rangle
_{V^{\prime }} &=&\int_{\mathcal{O}}\nabla \Psi \left( X\right) \cdot \nabla
\varphi _{j}d\xi -\int_{\mathcal{O}}KX\frac{\partial \varphi _{j}}{\partial
x_{3}}d\xi+\int_{\Gamma _{u}}\alpha Tr\left( \Psi \left( X\right) \right) Tr\left(
\varphi _{j}\right) d\sigma  \label{calc}
\end{eqnarray}%
where $\varphi _{j} =\lambda_j^{-1}e_j$ satisfies the equation 
\begin{equation*}
-\Delta \varphi _{j}=e_{j},\textnormal{ on }\mathcal{O},\ \dfrac{\partial \varphi _{j}}{\partial \nu }+\alpha \varphi _{j}=0, \textnormal{ on }\Gamma
_{u},\ 
\dfrac{\partial \varphi _{j}}{\partial \nu }=0, \textnormal{ on }\Gamma _{s}.
\end{equation*}
The first term on the right-hand side in (\ref{calc}) is treated using Green's formula to get 
\begin{eqnarray*}
\int_{\mathcal{O}}\nabla \Psi \left( X\right) \cdot \nabla \varphi _{j}d\xi
&=&-\int_{\mathcal{O}}\Psi \left( X\right) \Delta \varphi _{j}d\xi
+\int_{\Gamma _{u}}Tr\left( \Psi \left( X\right) \right) \nabla \varphi
_{j}\cdot \nu d\sigma \\
&=&\int_{\mathcal{O}}\Psi \left( X\right) e_{j}d\xi -\int_{\Gamma
_{u}}\alpha Tr\left( \Psi \left( X\right) \right) Tr\left( \varphi
_{j}\right) d\sigma
\end{eqnarray*}%
and, therefore,%
\begin{equation}
\left\langle A\left( X\left( r\right) \right) ,e_{j}\right\rangle
_{V^{\prime }}=\int_{\mathcal{O}}\left( \Psi \left( X\right) e_{j}-KX\frac{%
\partial \varphi _{j}}{\partial x_{3}}\right) d\xi.  \label{calc2}
\end{equation}

We can now give a precise definition of what the solution of our equation should be seen as.

\bigskip

\begin{definition}\label{DefSol}
Let us consider an initial condition $x\in V^{\prime }$. A $V^{\prime }-$valued continuous, $\mathbb{F}-$%
adapted stochastic process $X$ is called solution to equation (\ref{ecu}) on 
$\left[ 0,T\right] $ if $X\in L^{\infty }\left( 0,T;L^{2}\left( \Omega
;V^{\prime }\right) \right) $ and the following equality holds true for all the eigenfunctions $e_j$: 
\begin{eqnarray*}
&&\left\langle X\left( t\right) ,e_{j}\right\rangle _{V^{\prime }}+\int_{%
\mathcal{O}}\Psi \left( X\right) e_{j}d\xi -\int_{\mathcal{O}}KX\frac{%
\partial \varphi _{j}}{\partial x_{3}}d\xi
+\int_{0}^{t}\left\langle F_{u}\left( r\right) ,e_{j}\right\rangle
_{V^{\prime }}dr+\int_{0}^{t}\left\langle F_{s}\left( r\right)
,e_{j}\right\rangle _{V^{\prime }}dr \\
&=&\left\langle x,e_{j}\right\rangle _{V^{\prime }}+\int_{0}^{t}\left\langle
\sigma \left( X\left( r\right) \right) dW_{t},e_{j}\right\rangle _{V^{\prime
}},\quad \forall t\in \left[ 0,T\right] ,
\end{eqnarray*}
\end{definition}

$\mathbb{P}$-a.s. and for all $e_{j}$ from the orthonormal basis defined
above and where $\varphi _{j}$ satisfies the equation
\begin{equation}
-\Delta \varphi _{j}=e_{j},\textnormal{ on }\mathcal{O},\ \dfrac{\partial \varphi _{j}}{\partial \nu }+\alpha \varphi _{j}=0, \textnormal{ on }\Gamma
_{u},\ 
\dfrac{\partial \varphi _{j}}{\partial \nu }=0, \textnormal{ on }\Gamma _{s}.
\label{aprox}
\end{equation}
\subsection{Quasi-m-accretivity of the operator $A$}\label{Subsec2.4}
Before proceeding to the main result, we look into monotonicity properties of the operator $A$ previously defined.
\begin{lemma}
\label{lema1}Under the previous assumptions, the operator $A$ is quasi
m-accretive in $V^{\prime }$ provided that $C_0>0$. When $K=C_0=0$,the operator remains quasi-accretive.
\end{lemma}
\begin{proof}
Let $\mu$ be a positive real number which is assumed to be large enough. We aim at proving that
\begin{equation*}
\left\langle \left( \mu I+A\right) x-\left( \mu I+A\right)
y,x-y\right\rangle _{V^{\prime }}\geq 0,\text{ for all }x,y\in V^{\prime }.
\end{equation*}%
When $C_0>0$, we further prove that the range of the operator $\mu I+A$ is full, i.e.,
\begin{equation*}
\mathcal{R}\left( \mu I+A\right) =V^{\prime }.
\end{equation*}%
We denoted by $I$ the identity operator on $V'$ and by $\mathcal{R}$ the range of the operator 
$\mu I+A$.

We start with the accretivity of the operator, i.e.%
\begin{eqnarray*}
\left\langle \left( \mu I+A\right) x-\left( \mu I+A\right)
y,x-y\right\rangle _{V^{\prime }}
&=&\mu \left\vert x-y\right\vert _{V^{\prime }}^{2}+\int_{\mathcal{O}}\left(
\nabla \Psi \left( x\right) -\nabla \Psi \left( y\right) \right) \cdot
\nabla \varphi d\xi\\ &&-K\int_{\mathcal{O}}\left( x-y\right) \frac{\partial
\varphi }{\partial x_{3}}d\xi +\int_{\Gamma _{u}}\alpha \left( \Psi \left( x\right) -\Psi \left(
y\right) \right) \varphi d\sigma,
\end{eqnarray*}%
where%
\begin{equation}
-\Delta \varphi =x-y,\textnormal{ on }\mathcal{O},
\ \dfrac{\partial \varphi }{\partial \nu }+\alpha \varphi =0,\textnormal{ on }\Gamma _{u},\ \dfrac{\partial \varphi }{\partial \nu }=0,\textnormal{ on } \Gamma _{s}. \label{stea}
\end{equation}

Using Green's formula, we get 
\begin{eqnarray*}
\int_{\mathcal{O}}\left( \nabla \Psi \left( x\right) -\nabla \Psi \left(
y\right) \right) \cdot \nabla \varphi d\xi 
&=&-\int_{\mathcal{O}}\left( \Psi \left( x\right) -\Psi \left( y\right)
\right) \Delta \varphi d\xi +\int_{\Gamma _{u}}\left( \Psi \left( x\right)
-\Psi \left( y\right) \right) \dfrac{\partial \varphi }{\partial \nu }d\sigma
\\
&=&\int_{\mathcal{O}}\left( \Psi \left( x\right) -\Psi \left( y\right)
\right) (x-y) d\xi -\int_{\Gamma _{u}}\alpha \left( \Psi \left(
x\right) -\Psi \left( y\right) \right) \varphi d\sigma.
\end{eqnarray*}%
This further yields  
\begin{eqnarray*}
&&\left\langle \left( \mu I+A\right) x-\left( \mu I+A\right)
y,x-y\right\rangle _{V^{\prime }} \\
&=&\mu \left\vert x-y\right\vert _{V^{\prime }}^{2}+\int_{\mathcal{O}}\left(
\Psi \left( x\right) -\Psi \left( y\right) \right) (x-y) d\xi
-K\int_{\mathcal{O}}\left( x-y\right) \frac{\partial \varphi }{\partial x_{3}%
}d\xi \\
&\geq &\mu \left\vert x-y\right\vert _{V^{\prime }}^{2}+C_{0}\left\vert
x-y\right\vert _{L^{2}\left( \mathcal{O}\right) }^{2}-K\left\vert
x-y\right\vert _{L^{2}\left( \mathcal{O}\right) }\left\vert \frac{\partial
\varphi }{\partial x_{3}}\right\vert _{L^{2}\left( \mathcal{O}\right) }.
\end{eqnarray*}
One easily notes that
\begin{equation}
\left\vert \frac{\partial \varphi }{\partial x_{3}}\right\vert _{L^{2}\left( 
\mathcal{O}\right) }=\left\vert \left\langle \nabla \varphi
,i_{3}\right\rangle _{\mathbb{R}^{3}}\right\vert _{L^{2}\left( \mathcal{O}%
\right) }\leq \left\vert \nabla \varphi \right\vert _{L^{2}\left( \mathcal{O%
}\right) }.  \label{patrat}
\end{equation}
On the other hand, from (\ref{stea}), it follows that 
\begin{equation*}
\left\vert x-y\right\vert _{V^{\prime }}^{2}=\left\vert -\Delta \varphi
\right\vert _{V^{\prime }}^{2}=\left( -\Delta \varphi \right) \left( \gamma
\right) =\left\langle -\Delta \varphi ,\gamma \right\rangle _{2},
\end{equation*}%
where $\gamma $ is the solution of%
\begin{equation*}
-\Delta \gamma =-\Delta \varphi ,\textnormal{ on }\mathcal{O},\ \dfrac{\partial \gamma }{\partial \nu }+\alpha \gamma =0,\textnormal{ on } \Gamma _{u},\ 
\dfrac{\partial \gamma }{\partial \nu }=0,\textnormal{ on }\Gamma _{s}.%
\label{2stea}
\end{equation*}
Since $-\Delta \varphi =x-y,$ by the uniqueness of the solution of the previous equation, it follows that $\gamma =\varphi $.
Therefore, by invoking once again, Green's formula,%
\begin{equation*}
\left\vert x-y\right\vert _{V^{\prime }}^{2}=\left\langle -\Delta \varphi
,\varphi \right\rangle _{2}\geq \left\vert \nabla \varphi \right\vert
_{L^{2}\left( \mathcal{O}\right) }^{2}.
\end{equation*}
Going back to (\ref{patrat}), it follows that%
\begin{equation}
\left\vert \frac{\partial \varphi }{\partial x_{3}}\right\vert _{L^{2}\left( 
\mathcal{O}\right) }\leq \left\vert x-y\right\vert _{V^{\prime }},
\label{eess}
\end{equation}
and, as a consequence,
\begin{equation}\label{EstimQuasi}\begin{split}
\left\langle \left( \mu I+A\right) x-\left( \mu I+A\right)
y,x-y\right\rangle _{V^{\prime }}&\geq \mu \left\vert x-y\right\vert _{V^{\prime }}^{2}+C_{0}\left\vert
x-y\right\vert _{L^{2}\left( \mathcal{O}\right) }^{2}-K\left\vert
x-y\right\vert _{L^{2}\left( \mathcal{O}\right) }\left\vert x-y\right\vert
_{V^{\prime }}\\&\geq 0,\end{split}
\end{equation}%
for $\mu $ large enough. The reader is invited to note that the same holds true if $K=0$, for every $\mu>0$, and regardless of the value of $C_0\geq0$.
We shall prove now the maximality of the operator under the assumption that $C_0>0$. More precisely, we show
that for every $g\in V^{\prime }$ there exist $\theta \in D\left( A\right) $
which solve the equation 
\begin{equation}
\mu x+Ax=g.  \label{floare}
\end{equation}

Since $\Psi $\ is continuous and strongly monotone, increasing on $\left(
-\infty ,+\infty \right) $ and $R\left( \Psi \right) =\left( -\infty
,+\infty \right) $, one establishes that $\Psi ^{-1}$ is Lipschitz-continuous, and, therefore, it is
continuous from $V$ to $L^{2}\left( \mathcal{O}\right)$. (\ref{floare}) can be rewritten as 
$
\mu \Psi ^{-1}\left( y\right) +\overline{A}\left( y\right) =g,$
with $\overline{A}:V\rightarrow V^{\prime }$ defined by 
\begin{equation*}
_{V^{\prime }}\left\langle \overline{A}\left( y\right) ,\varphi
\right\rangle _{V}=\int_{\mathcal{O}}\nabla y\cdot \nabla \varphi d\xi
-\int_{\mathcal{O}}K\Psi ^{-1}\left( y\right) \frac{\partial \varphi }{%
\partial x_{3}}dx+\int_{\Gamma _{u}}\alpha \Psi ^{-1}\left( y\right) \varphi
d\sigma
\end{equation*}%
for $\forall \varphi \in V$.

By elementary computations, we can check that $\Psi ^{-1}+\overline{A}$\ is
continuous from $V$ to $V^{\prime }$, monotone and coercive. By using
Minty's theorem (see \cite{limita}) we have that $\Psi ^{-1}+\overline{A}$
is surjective which implies the existence of an unique solution to equation (%
\ref{floare}).
\end{proof}

\begin{remark}
We have chosen to give the proof here rather than postponing it to an appendix, since inequalities like \eqref{eess} and the interplay between $\mu$ large enough and $C_0$ and $K$ as emphasized in the proof of (quasi-)accretivity, i.e., through the inequality \eqref{EstimQuasi}, will play an important part in our subsequent arguments.
\end{remark}

\section{The well-posedness result}\label{Sec3}
We are now able to formulate and prove the main result of this paper.
\begin{theorem}\label{ThMain}
For every $x\in L^{2}\left( \mathcal{O}\right) ,$ the equation (\ref{ecu}) has a
unique solution in the sense of the Definition \ref{DefSol}, such that $X\in
C_{W}\left( \left[ 0,T\right] ;L^{2}\left( \Omega ;L^{2}\left( \mathcal{O}%
\right) \right) \right) $.
\end{theorem}

\bigskip

\begin{proof}
We have already established that the operator $A$ as defined above is quasi-accretive in $%
V^{\prime }$.
The equation (\ref{ecu}) can be rewritten, for $\mu>0$ large enough, in the
following equivalent form%
\begin{equation}
\left\{ 
\begin{array}{l}
dX\left( t\right) +\left( \mu I+A\right) \left( X\left( t\right) \right)
dt-\mu X\left( t\right) dt+F_{u}\left( t\right) dt+F_{s}\left( t\right)
dt\medskip \medskip  =\Sigma \left( X\left( t\right) \right) dW_{t},~t\in \left(
0,T\right), \\ 
X\left( 0\right) =x.%
\end{array}%
\right.  \label{echiv}
\end{equation}
It is obvious that a solution for $\mu>0$ is also a solution for some $\mu'=0$, and vice-versa. Furthermore, the accretivity property of $\mu I+A$ allows one to prove the uniqueness of the solution of \eqref{echiv}. Since these arguments are standard and follow from the inequality in \eqref{EstimQuasi}, we will omit the uniqueness proofs and concentrate on the \emph{existence} arguments.\\

\noindent We first approximate the operator $\Psi $ by $\widetilde{\Psi }%
_{\lambda }=\lambda I+\Psi _{\lambda }$ where $\Psi _{\lambda }$ is the
Yosida approximation of $\Psi $, i.e.%
\begin{equation*}
\Psi _{\lambda }\left( r\right) =\frac{1}{\lambda }\left( r-\left( I+\lambda
\Psi \right) ^{-1}\left( r\right) \right) =\Psi \left( \left( I+\lambda \Psi
\right) ^{-1}\left( r\right) \right) ,\quad \forall r\in \mathbb{R}\text{.}
\end{equation*}%
We denote the resolvent of $\Psi $ by $J_{\lambda }=\left( I+\lambda \Psi
\right) ^{-1}$. The reader is invited to note that, by a slight abuse of notation, we refer to $I$ as the identity function on $\mathbb{R}$, although, further on, this will also refer to the identity on $V'$.

Note that $\widetilde{\Psi }_{\lambda }$ is Lipschitz-continuous and strongly increasing in $\mathbb{R}$. Let us denote by $A_{\lambda }$ the
operator which is defined as $A$ where one replaces $\Psi $ by its
approximation $\widetilde{\Psi }_{\lambda }$. In particular, note that this operator is quasi m-accretif, owing to Lemma \ref{lema1}. Furthermore, the constant $\mu$ can be chosen large enough, and, prior to Step III, we will fix $\lambda>0$ and pick $\mu=\mu(\lambda)$.

This gives the first approximation of the equation (\ref{ecu})%
\begin{equation}
\left\{ 
\begin{array}{l}
dX_{\lambda }\left( t\right) +\left( \mu I+A_{\lambda }\right) \left(
X_{\lambda }\left( t\right) \right) dt-\mu X_{\lambda }\left( t\right)
dt+F_{u}\left( t\right) dt+F_{s}\left( t\right) dt=\Sigma \left( X_{\lambda }\left( t\right) \right)
dW_{t},~t\in \left( 0,T\right), \medskip \medskip \\ 
X\left( 0\right) =x.%
\end{array}%
\right.  \label{aprox1}
\end{equation}

\noindent In order to get the existence of the solution, we need another
approximation. By Lemma \ref{lema1} the operator $A_{\lambda }^{\mu }=\mu
I+A_{\lambda }$ is m-accretive in $V^{\prime }$ and, therefore, one can take
the Yosida approximation of the operator $A_{\lambda }^{\mu }$ in $V^{\prime
}$. For readers convenience, let us recall that 
\begin{equation*}
J_{\lambda }^{\mu, \varepsilon }=\left( I+\varepsilon A_{\lambda }^{\mu }\right)
^{-1},\quad \forall \varepsilon >0,
\end{equation*}%
is the resolvent of $A_{\lambda }^{\mu }$ and%
\begin{equation*}
A_{\lambda }^{\mu ,\varepsilon }=\frac{1}{\varepsilon }\left( I-J_{\lambda }^{\mu, \varepsilon }\right) =A_{\lambda }^{\mu }\left( J_{\lambda }^{\mu, \varepsilon }\right) ,\quad \forall \varepsilon >0,
\end{equation*}%
is the Yosida approximation of $A_{\lambda }^{\mu }$.

Since the operator $A_{\lambda }^{\mu ,\varepsilon }$ is Lipschitz in $%
V^{\prime }$, it follows by standard theory for stochastic
equations in Hilbert spaces that the approximating equation%
\begin{equation}
\left\{ 
\begin{array}{l}
dX_{\lambda }^{\varepsilon }\left( t\right) +A_{\lambda }^{\mu ,\varepsilon
}\left( X_{\lambda }^{\varepsilon }\left( t\right) \right) dt-\mu X_{\lambda
}^{\varepsilon }\left( t\right) dt+F_{u}\left( t\right) dt+F_{s}\left(
t\right) dt
=\sigma \left( X_{\lambda }^{\varepsilon }\left( t\right) \right)
dW_{t},~t\in \left( 0,T\right), \\ 
X\left( 0\right) =x,%
\end{array}%
\right.  \label{aprox2}
\end{equation}%
has a unique strong solution $X_{\lambda }^{\varepsilon }$ belonging to $C_{W}\left( %
\left[ 0,T\right] ;L^{2}\left( \Omega ;V^{\prime }\right) \right) $ with\\ $%
A_{\lambda }^{\mu ,\varepsilon }\left( X_{\lambda }^{\varepsilon }\left(
t\right) \right) \in C_{W}\left( \left[ 0,T\right] ;L^{2}\left( \Omega
;V^{\prime }\right) \right) $. We recall that the subscript $W$ specifies adaptedness with respect to the natural filtration, hence with respect to $\mathbb{F}$.

\bigskip

\noindent \textbf{Step I }(existence of the solution for the approximating equation in 
$L^{2}(\mathcal{O}$.)
In order to have also the existence of the solution in $C_{W}\left( \left[
0,T\right] ;L^{2}\left( \Omega ;L^{2}\left( \mathcal{O}\right) \right)
\right) ,$ we shall prove the following preliminary results.

\begin{lemma}\label{lema2}
The resolvent $y\mapsto J_{\lambda }^{\mu, \varepsilon }\left( y\right) =\left(
I+\varepsilon A_{\lambda }^{\mu }\right) ^{-1}\left( y\right) $ is Lipschitz
continuous in $L^{2}\left( \mathcal{O}\right) $.
\end{lemma}
We move the proof to the Appendix for our readers' sake.\\ 

\noindent Consequently $A_{\lambda }^{\mu ,\varepsilon }=\frac{1}{\varepsilon }\left(
I-J_{\lambda }^{\mu, \varepsilon }\right) $ is Lipschitz continuous in $%
L^{2}\left( \mathcal{O}\right) $ and, therefore, by standard existence
theory for stochastic PDEs, for each $x\in L^{2}\left( \mathcal{O}\right) $
the equation (\ref{aprox2}) has a unique solution $X_{\lambda }^{\varepsilon
}$ in $C_{W}\left( \left[ 0,T\right] ;L^{2}\left( \Omega ;L^{2}\left( 
\mathcal{O}\right) \right) \right) $ (see e.g. \cite{DPZ}).

\bigskip

\noindent \textbf{Step II} (convergence in $\varepsilon $). In this step, the $\varepsilon$ parameter is allowed to vanish, i.e., $\varepsilon \rightarrow 0$ and we look into the behavior of the limiting object.\\

\noindent To this purpose, the solution of the approximating
equation is written as 
\begin{equation}\label{ochi1}\begin{split}
&\left\langle X_{\lambda }^{\varepsilon }\left( t\right)
,e_{j}\right\rangle _{V^{\prime }}+\int_{0}^{t}\left\langle A_{\lambda
}^{\mu ,\varepsilon }\left( X_{\lambda }^{\varepsilon }\left( r\right)
\right) ,e_{j}\right\rangle _{V^{\prime }}dr \\
&-\mu \int_{0}^{t}\left\langle X_{\lambda }^{\varepsilon }\left( r\right)
,e_{j}\right\rangle _{V^{\prime }}dr+\int_{0}^{t}\left\langle F_{u}\left(
r\right) ,e_{j}\right\rangle _{V^{\prime }}dr+\int_{0}^{t}\left\langle
F_{s}\left( r\right) ,e_{j}\right\rangle _{V^{\prime }}dr\\
=&\left\langle x,e_{j}\right\rangle _{V^{\prime }}+\sum_{k=1}^{\infty }\mu
_{k}\int_{0}^{t}\left\langle X_{\lambda }^{\varepsilon }\left( r\right)
e_{k},e_{j}\right\rangle _{V^{\prime }}d\beta _{k}\left( r\right) ,\quad
\forall t\in \left[ 0,T\right] ,
\end{split}\end{equation}%
$\mathbb{P}$-a.s., and for all $e_{j}$ from the orthonormal basis formed by
the eigen-functions of the Robin-Laplace operator defined in (\ref{neumann}%
).

Simple computations yield, as before,
\begin{eqnarray*}
\left\langle A_{\lambda }^{\mu ,\varepsilon }\left( X_{\lambda
}^{\varepsilon }\right) ,e_{j}\right\rangle _{V^{\prime }}&=
&\left\langle \mu J_{\lambda }^{\mu, \varepsilon }\left( X_{\lambda
}^{\varepsilon }\right) ,e_{j}\right\rangle _{V^{\prime }}+\left\langle
A_{\lambda }\left( J_{\lambda }^{\mu, \varepsilon }\left( X_{\lambda
}^{\varepsilon }\right) \right) ,e_{j}\right\rangle _{V^{\prime }}\medskip \\
&=&\left\langle \mu J_{\lambda }^{\mu, \varepsilon }\left( X_{\lambda
}^{\varepsilon }\right) ,e_{j}\right\rangle _{V^{\prime }}+\int\limits_{%
\mathcal{O}}\widetilde{\Psi }_{\lambda }\left( J_{\lambda }^{\mu, \varepsilon
}\left( X_{\lambda }^{\varepsilon }\right) \right) e_{j}d\xi -\int\limits_{%
\mathcal{O}}KJ_{\lambda }^{\mu, \varepsilon }\left( X_{\lambda }^{\varepsilon
}\right) \frac{\partial \varphi _{j}}{\partial x_{3}}d\xi.
\end{eqnarray*}%
where $\varphi _{j}$ satisfies \eqref{aprox}.
The solution (\ref{ochi1}) of the approximating equation is rewritten as
follows

\begin{equation}\label{ochi2}\begin{split}
&\left\langle X_{\lambda }^{\varepsilon }\left( t\right)
,e_{j}\right\rangle _{V^{\prime }}+\mu \int_{0}^{t}\left\langle J_{\lambda
}^{\mu,\varepsilon }\left( X_{\lambda }^{\varepsilon }\left( r\right) \right)
,e_{j}\right\rangle _{V^{\prime }}dr \\
&+\lambda \int_{0}^{t}\int\limits_{\mathcal{O}}J_{\lambda }^{\mu,\varepsilon
}\left( X_{\lambda }^{\varepsilon }\left( r\right) \right) e_{j}d\xi
dr+\int_{0}^{t}\int\limits_{\mathcal{O}}\Psi _{\lambda }\left( J_{\lambda
}^{\mu,\varepsilon }\left( X_{\lambda }^{\varepsilon }\left( r\right) \right)
\right) e_{j}d\xi dr\\
&-K\int_{0}^{t}\int\limits_{\mathcal{O}}J_{\lambda }^{\mu, \varepsilon }\left(
X_{\lambda }^{\varepsilon }\left( r\right) \right) \frac{\partial \varphi
_{j}}{\partial x_{3}}d\xi dr-\mu \int_{0}^{t}\left\langle X_{\lambda
}^{\varepsilon }\left( r\right) ,e_{j}\right\rangle _{V^{\prime }}dr\\
&+\int_{0}^{t}\left\langle F_{u}\left( r\right) ,e_{j}\right\rangle
_{V^{\prime }}dr+\int_{0}^{t}\left\langle F_{s}\left( r\right)
,e_{j}\right\rangle _{V^{\prime }}dr\\
=&\left\langle x,e_{j}\right\rangle _{V^{\prime }}+\sum_{k=1}^{\infty }\mu
_{k}\int_{0}^{t}\left\langle X_{\lambda }^{\varepsilon }\left( r\right)
e_{k},e_{j}\right\rangle _{V^{\prime }}d\beta _{k}\left( r\right) ,\quad
\forall t\in \left[ 0,T\right] .  \end{split}
\end{equation}

Passing to the limit relies on the following preliminary
results.
\begin{lemma}\label{lema3}
For $x\in L^{2}\left( \mathcal{O}\right) $the following convergence results hold true as $\varepsilon\rightarrow 0$.
\begin{eqnarray*}
X_{\lambda }^{\varepsilon } &\rightharpoonup &X_{\lambda }\text{ weakly in }%
L^{\infty }\left( 0,T;L^{2}\left( \Omega ;L^{2}\left( \mathcal{O}\right)
\right) \right), \\
X_{\lambda }^{\varepsilon } &\rightarrow &X_{\lambda }\text{ strongly in }%
L^{\infty }\left( 0,T;L^{2}\left( \Omega ;V^{\prime }\right) \right), \\
\widetilde{\Psi }_{\lambda }\left( J_{\lambda }^{\mu, \varepsilon }\left(
X_{\lambda }^{\varepsilon }\right) \right) &\rightarrow &\eta \text{
strongly in }L^{2}\left( 0,T;L^{2}\left( \Omega ;L^{2}\left( \mathcal{O}%
\right) \right) \right),
\end{eqnarray*}%
where $X_{\lambda }$ is a solution to the equation (\ref{aprox1}). Furthermore, for a constant $C(\lambda)$ independent of $\varepsilon>0$,
\begin{equation}
\mathbb{E}\int_{0}^{T}\left\vert A_{\lambda }^{\mu ,\varepsilon }\left(
X_{\lambda }^{\varepsilon }\left( r\right) \right) \right\vert _{V^{\prime
}}^{2}dr\leq C\left( \lambda \right) .  \label{A}
\end{equation}
\end{lemma}
To facilitate the reading, the proof of the Lemma is postponed after the Theorem.

In order to pass to the limit in (\ref{ochi2}) we have first to note that 
\begin{equation*}
\left\vert X_{\lambda }^{\varepsilon }-J_{\lambda }^{\mu, \varepsilon }\left(
X_{\lambda }^{\varepsilon }\right) \right\vert _{V^{\prime }}^{2}\leq
\varepsilon \left\vert A_{\lambda }^{\mu ,\varepsilon }\left( X_{\lambda
}^{\varepsilon }\right) \right\vert _{V^{\prime }}^{2}
\end{equation*}%
which implies, by (\ref{A}), that 
\begin{equation*}
J_{\lambda }^{\mu, \varepsilon }\left( X_{\lambda }^{\varepsilon }\right)
\rightarrow X_{\lambda }\text{ strongly in }L^{2}\left( 0,T;L^{2}\left(
\Omega ;V^{\prime }\right) \right) .
\end{equation*}
On the other hand, one proves that%
\begin{equation*}
J_{\lambda }^{\mu, \varepsilon }\left( X_{\lambda }^{\varepsilon }\right)
\longrightarrow X_{\lambda }\text{ strongly in }L^{2}\left( 0,T;L^{2}\left(
\Omega ;L^{2}\left( \mathcal{O}\right) \right) \right),
\end{equation*}%
see (\ref{est7}) in the proof of Lemma \ref{lema3}.

Finally, since the map $\Psi _{\lambda }$ is maximal monotone in $%
L^{2}\left( 0,T;L^{2}\left( \Omega ;L^{2}\left( \mathcal{O}\right) \right)
\right) ,$ it is weakly-strongly closed and, therefore,%
\begin{equation*}
\eta =\Psi _{\lambda }\left( X_{\lambda }\right) .
\end{equation*}

One can now pass to the limit in (\ref{ochi2}) to get%
\begin{equation}\label{ochi3}\begin{split}
&\left\langle X_{\lambda }\left( t\right) ,e_{j}\right\rangle _{V^{\prime }}+\lambda \int_{0}^{t}\int\limits_{\mathcal{O}}X_{\lambda }\left( r\right)
e_{j}d\xi dr+\int_{0}^{t}\int\limits_{\mathcal{O}}\Psi _{\lambda }\left(
X_{\lambda }\left( r\right) \right) e_{j}d\xi dr\\
&-K\int_{0}^{t}\int\limits_{\mathcal{O}}X_{\lambda }\left( r\right) \frac{%
\partial \varphi _{j}}{\partial x_{3}}d\xi dr+\int_{0}^{t}\left\langle
F_{u}\left( r\right) ,e_{j}\right\rangle _{V^{\prime
}}dr+\int_{0}^{t}\left\langle F_{s}\left( r\right) ,e_{j}\right\rangle
_{V^{\prime }}dr\\
=&\left\langle x,e_{j}\right\rangle _{V^{\prime }}+\sum_{k=1}^{\infty }\mu
_{k}\int_{0}^{t}\left\langle X_{\lambda }\left( r\right)
e_{k},e_{j}\right\rangle _{V^{\prime }}d\beta _{k}\left( r\right) ,\quad
\forall t\in \left[ 0,T\right] ,
\end{split}\end{equation}%
where $\varphi _{j}$ satisfies \eqref{aprox}.\\
At this point, we emphasize, once again, that $\mu$ plays a purely fictitious part in the sense that the equation is the same with some $\mu'=0$ and, from now on, we concentrate on this latter formulation.\\

\noindent \textbf{Step III a.} (weak convergences as $\lambda\rightarrow 0 $)\\
In order to pass to the limit for $\lambda \rightarrow 0$, one applies It\^{o}'s formula with the function $\abs{\cdot}_{V'}$, on $[0,t]$ to the $V^{\prime }$-valued It\^{o} process $X_\lambda$, and get 
\begin{eqnarray*}
&&\mathbb{E}\frac{1}{2}\left\vert X_{\lambda }\left( t\right) \right\vert
_{V^{\prime }}^{2}+\lambda \mathbb{E}\int_{0}^{t}\left\vert X_{\lambda
}\left( r\right) \right\vert _{V^{\prime }}^{2}dr+\mathbb{E}%
\int_{0}^{t}\int_{\mathcal{O}}\Psi _{\lambda }\left( X_{\lambda }\left(
r\right) \right) X_{\lambda }\left( r\right) d\xi dr \\
&&-K\mathbb{E}\int_{0}^{t}\int_{\mathcal{O}}X_{\lambda }\left( r\right) 
\frac{\partial \varphi _{\lambda }}{\partial x_{3}}d\xi dr+\mathbb{E}%
\int_{0}^{t}\left\langle F_{u}\left( r\right) +F_{s}\left( r\right)
,X_{\lambda }\left( r\right) \right\rangle _{V^{\prime }}dr \\
&=&\frac{1}{2}\left\vert x\right\vert _{V^{\prime }}^{2}+\mathbb{E}%
\sum_{k=1}^{\infty }\mu _{k}\int_{0}^{t}\left\vert X_{\lambda }\left(
r\right) e_{k}\right\vert _{V^{\prime }}^{2}dr,
\end{eqnarray*}%
where $\varphi _{\lambda }$ is the solution to 
\begin{equation*}
-\Delta \varphi _{\lambda }=X_{\lambda }, \textnormal{ on } \mathcal{O}, \ 
\dfrac{\partial \varphi _{\lambda }}{\partial \nu }+\alpha \varphi _{\lambda
}=0, \textnormal{ on } \Gamma _{u}, \ \dfrac{\partial \varphi _{\lambda }}{\partial \nu }=0, \textnormal{ on } \Gamma _{s}.%
\end{equation*}
Keeping in mind that, for some constant $C$ independent of $\lambda $,%
\begin{eqnarray*}
\left\vert \frac{\partial \varphi _{\lambda }}{\partial x_{3}}\right\vert
_{2} &\leq &C\left\vert X_{\lambda }\right\vert _{V^{\prime }} \textnormal{ and }
\left\vert X_{\lambda }e_{k}\right\vert _{V^{\prime }} \leq C\\pr{1+\lambda_k^{\frac{d+1}{2}}}\left\vert
X_{\lambda }\right\vert _{V^{\prime }},
\end{eqnarray*}%
(see Appendix), and due to the assumptions on $F_{u}$ and $F_{s},$\ one has
\begin{equation}\label{CC}\begin{split}
&\mathbb{E}\frac{1}{2}\left\vert X_{\lambda }\left( t\right) \right\vert
_{V^{\prime }}^{2}+\lambda \mathbb{E}\int_{0}^{t}\left\vert X_{\lambda
}\left( r\right) \right\vert _{V^{\prime }}^{2}dr+\mathbb{E}\int_{0}^{t}\int_{\mathcal{O}}\Psi _{\lambda }\left( X_{\lambda
}\left( r\right) \right) X_{\lambda }\left( r\right) d\xi dr
\\
\leq& C\pr{1+\left\vert x\right\vert _{V^{\prime }}^{2}}+C\mathbb{E}%
\int_{0}^{t}\left\vert X_{\lambda }\left( r\right) \right\vert _{V^{\prime
}}^{2}dr+\frac{C_{4}}{4}\mathbb{E}\int_{0}^{t}\left\vert X_{\lambda }\left(
r\right) \right\vert _{2}^{2}dr.  
\end{split}
\end{equation}

On the other hand owing to the assumptions on $j(\cdot)$,
\begin{eqnarray*}
&&\mathbb{E}\int_{0}^{t}\int_{\mathcal{O}}\Psi _{\lambda }\left( X_{\lambda
}\right) X_{\lambda }d\xi dr=\mathbb{E}\int_{0}^{t}\int_{\mathcal{O}}\Psi \left( J_{\lambda }\left(
X_{\lambda }\right) \right) \left( X_{\lambda }+J_{\lambda }\left(
X_{\lambda }\right) -J_{\lambda }\left( X_{\lambda }\right) \right) d\xi
dr\\
&=&\mathbb{E}\int_{0}^{t}\int_{\mathcal{O}}\Psi \left( J_{\lambda }\left(
X_{\lambda }\right) \right) J_{\lambda }\left( X_{\lambda }\right) d\xi
dr+\mathbb{E}\int_{0}^{t}\int_{\mathcal{O}%
}\Psi _{\lambda }\left( X_{\lambda }\right) \left( X_{\lambda }-J_{\lambda
}\left( X_{\lambda }\right) \right) d\xi dr\medskip \\
&\geq &\mathbb{E}\int_{0}^{t}\int_{\mathcal{O}}\left( C_3\left\vert
J_{\lambda }\left( X_{\lambda }\right) \right\vert ^{m+1}+C_{4}\left\vert
J_{\lambda }\left( X_{\lambda }\right) \right\vert ^{2}{-C_5}\right) d\xi
dr+\frac{1}{\lambda }\mathbb{E}%
\int_{0}^{t}\int_{\mathcal{O}}\left\vert \left( X_{\lambda }-J_{\lambda
}\left( X_{\lambda }\right) \right) \right\vert ^{2}d\xi dr.
\end{eqnarray*}

Going back to (\ref{CC}), one gets
\begin{eqnarray*}
&&\mathbb{E}\frac{1}{2}\left\vert X_{\lambda }\left( t\right) \right\vert
_{V^{\prime }}^{2}+\lambda \mathbb{E}\int_{0}^{t}\left\vert X_{\lambda
}\left( r\right) \right\vert _{V^{\prime }}^{2}dr+\mathbb{E}\int_{0}^{t}\int_{\mathcal{O}}\left( C_{3}\left\vert J_{\lambda
}\left( X_{\lambda }\right) \right\vert ^{m+1}+C_{4}\left\vert J_{\lambda
}\left( X_{\lambda }\right) \right\vert ^{2}-C_5\right) d\xi dr\medskip \\
&&+\frac{1}{\lambda }\mathbb{E}\int_{0}^{t}\int_{\mathcal{O}}\left\vert
X_{\lambda }-J_{\lambda }\left( X_{\lambda }\right) \right\vert ^{2}d\xi
dr\medskip \\
&\leq &C\pr{1+\left\vert x\right\vert _{V^{\prime }}^{2}}+C\mathbb{E}%
\int_{0}^{t}\left\vert X_{\lambda }\left( r\right) \right\vert _{V^{\prime
}}^{2}dr \\
&&+\frac{C_{4}}{2}\mathbb{E}\int_{0}^{t}\int_{\mathcal{O}}\left\vert
X_{\lambda }\left( r\right) -J_{\lambda }\left( X_{\lambda }\right)
\right\vert ^{2}d\xi dr+\frac{C_{4}}{2}\mathbb{E}\int_{0}^{t}\int_{\mathcal{O%
}}\left\vert J_{\lambda }\left( X_{\lambda }\right) \right\vert ^{2}d\xi dr.
\end{eqnarray*}

For $\lambda $ sufficiently small and owing to Gronwall's inequality,
\begin{eqnarray*}
&&\mathbb{E}\frac{1}{2}\left\vert X_{\lambda }\left( t\right) \right\vert
_{V^{\prime }}^{2}+\lambda \mathbb{E}\int_{0}^{t}\left\vert X_{\lambda
}\left( r\right) \right\vert _{V^{\prime }}^{2}dr+\mathbb{E}\int_{0}^{t}\int_{\mathcal{O}}\left( C_{3}\left\vert J_{\lambda
}\left( X_{\lambda }\right) \right\vert ^{m+1}+\frac{C_{4}}{2}\left\vert
J_{\lambda }\left( X_{\lambda }\right) \right\vert ^{2}\right) d\xi
dr\medskip  \\
&&+\left( \frac{1}{\lambda }-\frac{C_{4}}{2}\right) \mathbb{E}%
\int_{0}^{t}\int_{\mathcal{O}}\left\vert X_{\lambda }-J_{\lambda }\left(
X_{\lambda }\right) \right\vert ^{2}d\xi dr\leq C\pr{1+\left\vert x\right\vert_{V^{\prime }}^{2}},
\end{eqnarray*}%
and, therefore,%
\begin{eqnarray*}
&&\left\{ X_{\lambda }\right\} \text{ is bounded in }L^{\infty }\left(
0,T;L^{2}\left( \Omega ;V^{\prime }\right) \right) \medskip  \\
&&\left\{ J_{\lambda }\left( X_{\lambda }\right) \right\} \text{ is bounded
in }L^{m+1}\left( \left( 0,T\right) \times \Omega \times \mathcal{O}\right) .
\end{eqnarray*}
From the assumptions on the upper bounds on $\Psi $,
\begin{equation*}
\left\{ \Psi \left( J_{\lambda }\left( X_{\lambda }\right) \right) \right\} 
\text{ is bounded in }L^{\frac{m+1}{m}}\left( \left( 0,T\right) \times
\Omega \times \mathcal{O}\right) .
\end{equation*}
From the upper bounds on $\mathbb{E}%
\int_{0}^{t}\int_{\mathcal{O}}\left\vert X_{\lambda }-J_{\lambda }\left(
X_{\lambda }\right) \right\vert ^{2}d\xi dr$, one deduces that
\begin{equation*}
\left\{ X_{\lambda }\right\} \text{ is bounded in }L^{2}\left( \left(
0,T\right) \times \Omega \times \mathcal{O}\right) .
\end{equation*}
To summarize, one establishes the following weak convergences
\begin{eqnarray*}
X_{\lambda } &\rightharpoonup &X\text{ weakly in }L^{\infty }\left(
0,T;L^{2}\left( \Omega ;V^{\prime }\right) \right), \\
&&\quad \quad \quad \quad \text{and }L^{2}\left( \left( 0,T\right) \times
\Omega \times \mathcal{O}\right), \\
J_{\lambda }\left( X_{\lambda }\right) &\rightharpoonup &X\text{ weakly in }%
L^{m+1}\left( \left( 0,T\right) \times \Omega \times \mathcal{O}\right), \\
\Psi \left( J_{\lambda }\left( X_{\lambda }\right) \right) &\rightharpoonup
&\eta \text{ weakly in }L^{\frac{m+1}{m}}\left( \left( 0,T\right) \times
\Omega \times \mathcal{O}\right) .
\end{eqnarray*}

\noindent \textbf{Step III b.} (strong convergence in $L^{\infty }\left( 0,T;L^{2}\left( \Omega
;V^{\prime }\right) \right)$ and conclusion.)\\
In order to conclude the proof one still has to show that the limit of $\Psi
\left( J_{\lambda }\left( X_{\lambda }\right) \right) $ can be identified with $\Psi \left(
X\right) $.

\noindent Since the operator $\Psi $ is maximal monotone in the duality pair 
\begin{equation*}
\pr{L^{m+1}\left( \left( 0,T\right) \times \Omega \times \mathcal{O}\right)
, L^{\frac{m+1}{m}}\left( \left( 0,T\right) \times \Omega \times 
\mathcal{O}\right) },
\end{equation*}%
it is sufficient to show that 
\begin{equation*}
\underset{\lambda \rightarrow 0}{\lim \inf }\ \mathbb{E}\int_{0}^{t}\int_{%
\mathcal{O}}\Psi \left( J_{\lambda }\left( X_{\lambda }\right) \right)
J_{\lambda }\left( X_{\lambda }\right) d\xi dr\leq \mathbb{E}%
\int_{0}^{t}\int_{\mathcal{O}}\eta Xd\xi dr.
\end{equation*}

To this end, it suffices to show the strong convergence of $\left\{
X_{\lambda }\right\} $ in $L^{\infty }\left( 0,T;L^{2}\left( \Omega
;V^{\prime }\right) \right) .$ One employs It\^{o}'s formula for the squared norm in $V^{\prime }$ on $[0,t]$, and gets 
\begin{eqnarray*}
&&\mathbb{E}\frac{1}{2}\left\vert X_{\lambda }\left( t\right) -X_{\lambda
^{\prime }}\left( t\right) \right\vert _{V^{\prime }}^{2} +\mathbb{E}\int_{0}^{t}\left\langle \lambda X_{\lambda }\left( r\right)
-\lambda ^{\prime }X_{\lambda ^{\prime }}\left( r\right) ,X_{\lambda }\left(
r\right) -X_{\lambda ^{\prime }}\left( r\right) \right\rangle _{V^{\prime
}}dr\medskip \\
&&+\mathbb{E}\int_{0}^{t}\int_{\mathcal{O}}\left( \Psi _{\lambda }\left(
X_{\lambda }\left( r\right) \right) -\Psi _{\lambda ^{\prime }}\left(
X_{\lambda ^{\prime }}\left( r\right) \right) \right) \left( X_{\lambda
}\left( r\right) -X_{\lambda ^{\prime }}\left( r\right) \right) d\xi dr \\
&&-K\mathbb{E}\int_{0}^{t}\int_{\mathcal{O}}\left( X_{\lambda }\left(
r\right) -X_{\lambda ^{\prime }}\left( r\right) \right) \frac{\partial
\varphi _{\lambda ,\lambda ^{\prime }}}{\partial x_{3}}d\xi dr \\
&&+\mathbb{E}\int_{0}^{t}\left\langle F_{u}\left( r\right) +F_{s}\left(
r\right) ,X_{\lambda }\left( r\right) -X_{\lambda ^{\prime }}\left( r\right)
\right\rangle _{V^{\prime }}dr =\mathbb{E}\sum_{k=1}^{\infty }\mu _{k}\int_{0}^{t}\left\vert \left(
X_{\lambda }\left( r\right) -X_{\lambda ^{\prime }}\left( r\right) \right)
e_{k}\right\vert _{V^{\prime }}^{2}dr,
\end{eqnarray*}%
where $\varphi _{\lambda,\lambda' }$ is the solution to 
\begin{equation*}
-\Delta \varphi _{\lambda ,\lambda ^{\prime }}=X_{\lambda }-X_{\lambda
^{\prime }}, \textnormal{ on } \mathcal{O}, \
\dfrac{\partial \varphi _{\lambda ,\lambda ^{\prime }}}{\partial \nu }%
+\alpha \varphi _{\lambda ,\lambda ^{\prime }}=0, \textnormal{ on } \Gamma _{u}, \ 
\dfrac{\partial \varphi _{\lambda ,\lambda ^{\prime }}}{\partial \nu }=0, \textnormal{ on }
\Gamma _{s}.
\end{equation*}

By recalling that, for some constant $C$ independent of $\lambda $,
\begin{eqnarray*}
\left\vert \frac{\partial \varphi _{\lambda ,\lambda ^{\prime }}}{\partial
x_{3}}\right\vert _{2} \leq C\left\vert X_{\lambda }-X_{\lambda ^{\prime
}}\right\vert _{V^{\prime }}\textnormal{ and }
\left\vert \left( X_{\lambda }-X_{\lambda ^{\prime }}\right)
e_{k}\right\vert _{V^{\prime }} \leq C\pr{1+\lambda_k^{\frac{d+1}{2}}}\left\vert X_{\lambda }-X_{\lambda
^{\prime }}\right\vert _{V^{\prime }},
\end{eqnarray*}%
and from the assumptions on $F_{u}$ and $F_{s},$\ it follows that
\begin{equation}\label{calc9}\begin{split}
&\mathbb{E}\frac{1}{2}\left\vert X_{\lambda }\left( t\right) -X_{\lambda
^{\prime }}\left( t\right) \right\vert _{V^{\prime }}^{2}+\mathbb{E}\int_{0}^{t}\left\langle \lambda X_{\lambda }\left( r\right)
-\lambda ^{\prime }X_{\lambda ^{\prime }}\left( r\right) ,X_{\lambda }\left(
r\right) -X_{\lambda ^{\prime }}\left( r\right) \right\rangle _{V^{\prime
}}dr\\
&+\mathbb{E}\int_{0}^{t}\int_{\mathcal{O}}\left( \Psi _{\lambda }\left(
X_{\lambda }\left( r\right) \right) -\Psi _{\lambda ^{\prime }}\left(
X_{\lambda ^{\prime }}\left( r\right) \right) \right) \left( X_{\lambda
}\left( r\right) -X_{\lambda ^{\prime }}\left( r\right) \right) d\xi
dr\\
\leq &C\mathbb{E}\int_{0}^{t}\left\vert X_{\lambda }\left( r\right)
-X_{\lambda ^{\prime }}\left( r\right) \right\vert _{V^{\prime }}^{2}dr+%
\frac{C_{0}}{8}\mathbb{E}\int_{0}^{t}\left\vert X_{\lambda }\left( r\right)
-X_{\lambda ^{\prime }}\left( r\right) \right\vert _{2}^{2}dr.  \end{split}
\end{equation}
As before
\begin{eqnarray*}
&&\left( \Psi _{\lambda }\left( X_{\lambda }\right) -\Psi _{\lambda ^{\prime
}}\left( X_{\lambda ^{\prime }}\right) \right) \left( X_{\lambda
}-X_{\lambda ^{\prime }}\right) \medskip  \\
&=&\left( \Psi _{\lambda }\left( X_{\lambda }\right) -\Psi _{\lambda
^{\prime }}\left( X_{\lambda ^{\prime }}\right) \right) \left( J_{\lambda
}\left( X_{\lambda }\right) -J_{\lambda ^{\prime }}\left( X_{\lambda
^{\prime }}\right) \right) \medskip  \\
&&+\left( \Psi _{\lambda }\left( X_{\lambda }\right) -\Psi _{\lambda
^{\prime }}\left( X_{\lambda ^{\prime }}\right) \right) \left( \lambda \Psi
_{\lambda }\left( X_{\lambda }\right) -\lambda ^{\prime }\Psi _{\lambda
^{\prime }}\left( X_{\lambda ^{\prime }}\right) \right) \medskip  \\
&\geq &C_{0}\left\vert J_{\lambda }\left( X_{\lambda }\right) -J_{\lambda
^{\prime }}\left( X_{\lambda ^{\prime }}\right) \right\vert ^{2}-2\left( \lambda +\lambda ^{\prime }\right) \left( \left\vert \Psi
_{\lambda }\left( X_{\lambda }\right) \right\vert ^{2}+\left\vert \Psi
_{\lambda ^{\prime }}\left( X_{\lambda ^{\prime }}\right) \right\vert
^{2}\right) ,
\end{eqnarray*}%
and 
\begin{eqnarray*}
\left\vert X_{\lambda }\left( r\right) -X_{\lambda ^{\prime }}\left(
r\right) \right\vert ^{2} &\leq &3\left\vert X_{\lambda }-J_{\lambda }\left(
X_{\lambda }\right) \right\vert ^{2}+3\left\vert X_{\lambda ^{\prime
}}-J_{\lambda ^{\prime }}\left( X_{\lambda ^{\prime }}\right) \right\vert
^{2}+3\left\vert J_{\lambda }\left( X_{\lambda }\right) -J_{\lambda ^{\prime
}}\left( X_{\lambda ^{\prime }}\right) \right\vert ^{2}\medskip  \\
&=&3\lambda^2 \left\vert \Psi _{\lambda }\left( X_{\lambda }\right)
\right\vert ^{2}+3\pr{\lambda ^{\prime }}^2\left\vert \Psi _{\lambda ^{\prime
}}\left( X_{\lambda ^{\prime }}\right) \right\vert ^{2}+3\left\vert J_{\lambda }\left( X_{\lambda }\right) -J_{\lambda ^{\prime
}}\left( X_{\lambda ^{\prime }}\right) \right\vert ^{2}.
\end{eqnarray*}
Going back to (\ref{calc9}), one concludes, for $0<\lambda,\lambda'<1$,%
\begin{eqnarray*}
&&\mathbb{E}\frac{1}{2}\left\vert X_{\lambda }\left( t\right) -X_{\lambda
^{\prime }}\left( t\right) \right\vert _{V^{\prime }}^{2}+\mathbb{E}\int_{0}^{t}\int_{\mathcal{O}}\frac{C_{0}}{2}\left\vert
J_{\lambda }\left( X_{\lambda }\right) -J_{\lambda ^{\prime }}\left(
X_{\lambda ^{\prime }}\right) \right\vert ^{2}d\xi dr\medskip  \\
&\leq &C\mathbb{E}\int_{0}^{t}\left\vert X_{\lambda }\left( r\right)
-X_{\lambda ^{\prime }}\left( r\right) \right\vert _{V^{\prime
}}^{2}dr\medskip  \\
&&+C\pr{\lambda ^{\prime }+\lambda}\mathbb{E}\int_{0}^{t}\pp{\int_{\mathcal{O}}\pr{\left\vert
\Psi _{\lambda ^{\prime }}\left( X_{\lambda ^{\prime }}\right) \right\vert
^{2}+\left\vert
\Psi _{\lambda }\left( X_{\lambda }\right) \right\vert
^{2}}d\xi +
\left\vert X_{\lambda }\left( r\right) \right\vert _{V^{\prime
}}^{2}+\left\vert X_{\lambda ^{\prime }}\left( r\right) \right\vert
_{V^{\prime }}^{2}}dr.
\end{eqnarray*}%
and the strong convergence follows from Gronwall's inequality and the aforementioned estimates for the right-hand terms.

Finally, by the same argument as in \cite{criticality}, we get that $\eta
=\Psi \left( X\right) $ and the proof is complete.
\end{proof}\\

In the previous result, during the second Step, we have used the convergence stated in Lemma \ref{lema3}. Let us now prove those statements.\\

\begin{proof}[Proof of Lemma \ref{lema3}]
\textbf{Step 1 (estimates in $V'$).} First, we apply It\^{o}'s formula to the function $\abs{\cdot}_{V'}^2$ on $\pp{0,T}$ with the $V'$-valued diffusion $X_\lambda^\varepsilon$. We get %
\begin{eqnarray*}
&&\mathbb{E}\left\vert X_{\lambda }^{\varepsilon }\left( t\right)
\right\vert _{V^{\prime }}^{2}+2\mathbb{E}\int_{0}^{t}\left\langle
A_{\lambda }^{\mu ,\varepsilon }\left( X_{\lambda }^{\varepsilon }\left(
r\right) \right) ,X_{\lambda }^{\varepsilon }\left( r\right) \right\rangle
_{V^{\prime }}dr-2\mu \mathbb{E}\int_{0}^{t}\left\vert X_{\lambda
}^{\varepsilon }\left( r\right) \right\vert _{V^{\prime }}^{2}dr \\
&&+2\mathbb{E}\int_{0}^{t}\left\langle F_{u}\left( r\right) ,X_{\lambda
}^{\varepsilon }\left( r\right) \right\rangle_{V'} dr+2\mathbb{E}%
\int_{0}^{t}\left\langle F_{s}\left( r\right) ,X_{\lambda }^{\varepsilon
}\left( r\right) \right\rangle_{V'} dr \leq \left\vert x\right\vert _{V^{\prime }}^{2}+C\mathbb{E}%
\int_{0}^{t}\left\vert X_{\lambda }^{\varepsilon }\left( r\right)
\right\vert _{V^{\prime }}^{2}dr.
\end{eqnarray*}
On the one hand,
\begin{eqnarray*}
\left\langle A_{\lambda }^{\mu ,\varepsilon }\left( X_{\lambda
}^{\varepsilon }\right) ,X_{\lambda }^{\varepsilon }\right\rangle
_{V^{\prime }}
&=&\left\langle A_{\lambda }^{\mu }\left( J_{\lambda }^{\mu, \varepsilon }\left(
X_{\lambda }^{\varepsilon }\right) \right) ,J_{\lambda }^{\mu,\varepsilon
}\left( X_{\lambda }^{\varepsilon }\right) \right\rangle _{V^{\prime }}+%
\frac{1}{\varepsilon }\left\vert X_{\lambda }^{\varepsilon }-J_{\lambda
}^{\mu,\varepsilon }\left( X_{\lambda }^{\varepsilon }\right) \right\vert
_{V^{\prime }}^{2} \\
&=&\mu \left\vert J_{\lambda }^{\mu, \varepsilon }\left( X_{\lambda
}^{\varepsilon }\right) \right\vert _{V^{\prime }}^{2}+\left\langle
A_{\lambda }\left( J_{\lambda }^{\mu, \varepsilon }\left( X_{\lambda
}^{\varepsilon }\right) \right) ,J_{\lambda }^{\mu, \varepsilon }\left(
X_{\lambda }^{\varepsilon }\right) \right\rangle _{V^{\prime }}+\frac{1}{%
\varepsilon }\left\vert X_{\lambda }^{\varepsilon }-J_{\lambda
}^{\mu,\varepsilon }\left( X_{\lambda }^{\varepsilon }\right) \right\vert
_{V^{\prime }}^{2}\geq 0.
\end{eqnarray*}
By using Gronwall's inequality, one has%
\begin{equation}
\mathbb{E}\left\vert X_{\lambda }^{\varepsilon }\left( t\right) \right\vert
_{V^{\prime }}^{2}+2\mathbb{E}\int_{0}^{t}\mu \left\vert J_{\lambda
}^{\mu,\varepsilon }\left( X_{\lambda }^{\varepsilon }\left( r\right) \right)
\right\vert _{V^{\prime }}^{2}dr+2\mathbb{E}\int_{0}^{t}\frac{1}{\varepsilon 
}\left\vert X_{\lambda }^{\varepsilon }\left( r\right) -J_{\lambda
}^{\mu,\varepsilon }\left( X_{\lambda }^{\varepsilon }\left( r\right) \right)
\right\vert _{V^{\prime }}^{2}dr\leq C\pr{1+\abs{x}_{V'}^2},  \label{Vprim}
\end{equation}%
for $\forall t\in \left[ 0,T\right],$ where $C$ is a constant independent
of $\varepsilon $ and $\lambda $.

\textbf{Step 2 (estimates in $L^2(\mathcal{O})$).} We continue with applying It\^{o}'s formula to the function $\abs{\cdot}_2:=\abs{\cdot}_{L^{2}\pr{\mathcal{O}}}$, to the $L^{2}\left( \mathcal{O}%
\right) $-diffusion $X_\lambda^\varepsilon$ to get
\begin{equation*}\begin{split}
\mathbb{E}\left\vert X_{\lambda }^{\varepsilon }\left( t\right) \right\vert
_{2}^{2}+2\mathbb{E}\int_{0}^{t}~_{V^{\prime }}\left\langle A_{\lambda
}^{\mu ,\varepsilon }\left( X_{\lambda }^{\varepsilon }\left( r\right)
\right) ,X_{\lambda }^{\varepsilon }\left( r\right) \right\rangle
_{V}dr-2\mu \mathbb{E}\int_{0}^{t}\left\vert X_{\lambda }^{\varepsilon
}\left( r\right) \right\vert _{2}^{2}dr\\
+2\mathbb{E}\int_{0}^{t}\left\langle F_{u}\left( r\right) +F_{s}\left(
r\right) ,X_{\lambda }^{\varepsilon }\left( r\right) \right\rangle
_{2}dr\leq \left\vert x\right\vert _{2}^{2}+C\mathbb{E}\int_{0}^{t}\left%
\vert X_{\lambda }^{\varepsilon }\left( r\right) \right\vert _{2}^{2}dr.\end{split}
\end{equation*}
Taking into account that 
\begin{eqnarray*}
_{V^{\prime }}\left\langle A_{\lambda }^{\mu ,\varepsilon }\left(
X_{\lambda }^{\varepsilon }\right) ,X_{\lambda }^{\varepsilon }\right\rangle
_{V}
&=&_{V^{\prime }}\left\langle A_{\lambda }^{\mu }\left( J_{\lambda
}^{\mu,\varepsilon }\left( X_{\lambda }^{\varepsilon }\right) \right)
,J_{\lambda }^{\mu, \varepsilon }\left( X_{\lambda }^{\varepsilon }\right)
+\varepsilon A_{\lambda }^{\mu }\left( J_{\lambda }^{\mu, \varepsilon }\left(
X_{\lambda }^{\varepsilon }\right) \right) \right\rangle _{V}\medskip \\
&=&\mu \left\vert J_{\lambda }^{\mu, \varepsilon }\left( X_{\lambda
}^{\varepsilon }\right) \right\vert _{2}^{2}+_{V^{\prime }}\left\langle
A_{\lambda }\left( J_{\lambda }^{\mu, \varepsilon }\left( X_{\lambda
}^{\varepsilon }\right) \right) ,J_{\lambda }^{\mu, \varepsilon }\left(
X_{\lambda }^{\varepsilon }\right) \right\rangle _{V}+\varepsilon \left\vert
A_{\lambda }^{\mu }\left( J_{\lambda }^{\mu, \varepsilon }\left( X_{\lambda
}^{\varepsilon }\right) \right) \right\vert _{2}^{2},
\end{eqnarray*}%
Gronwall's inequality yields that 
\begin{equation}
\mathbb{E}\left\vert X_{\lambda }^{\varepsilon }\left( t\right) \right\vert
_{2}^{2}+2\mathbb{E}\int_{0}^{t}\left( \mu \left\vert J_{\lambda
}^{\mu,\varepsilon }\left( X_{\lambda }^{\varepsilon }\left( r\right) \right)
\right\vert _{2}^{2}+_{V^{\prime }}\left\langle A_{\lambda }\left(
J_{\lambda }^{\mu, \varepsilon }\left( X_{\lambda }^{\varepsilon }\left( r\right)
\right) \right) ,J_{\lambda }^{\mu, \varepsilon }\left( X_{\lambda }^{\varepsilon
}\left( r\right) \right) \right\rangle _{V}\right) dr\leq C\pr{1+\abs{x}_2^2},  \label{Ldoi}
\end{equation}%
for $\forall t\in \left[ 0,T\right] ,$ where $C$ is a constant independent
of $\varepsilon $ and $\lambda $.

\textbf{Step 3 (estimates of $\mathbb{E}\int_{0}^{t}\left\vert A_{\lambda }^{\mu ,\varepsilon }\left(
X_{\lambda }^{\varepsilon }\left( r\right) \right) \right\vert _{V^{\prime
}}^{2}dr$).}

One begins with writing
\begin{eqnarray*}
\left\vert A_{\lambda }^{\mu ,\varepsilon }\left( X_{\lambda }^{\varepsilon
}\left( r\right) \right) \right\vert _{V^{\prime }}^{2} 
&=&\left\langle \mu J_{\lambda }^{\mu, \varepsilon }\left( X_{\lambda
}^{\varepsilon }\right) +A_{\lambda }\left( J_{\lambda }^{\mu,\varepsilon
}\left( X_{\lambda }^{\varepsilon }\right) \right) ,\mu J_{\lambda
}^{\mu,\varepsilon }\left( X_{\lambda }^{\varepsilon }\right) +A_{\lambda
}\left( J_{\lambda }^{\mu, \varepsilon }\left( X_{\lambda }^{\varepsilon }\right)
\right) \right\rangle _{V^{\prime }}\medskip \\
&\leq &2\mu \left\vert J_{\lambda }^{\mu, \varepsilon }\left( X_{\lambda
}^{\varepsilon }\right) \right\vert _{V^{\prime }}^{2}+2\left\vert
A_{\lambda }\left( J_{\lambda }^{\mu, \varepsilon }\left( X_{\lambda
}^{\varepsilon }\right) \right) \right\vert _{V^{\prime }}^{2}.
\end{eqnarray*}

\noindent Furthermore,
\begin{eqnarray*}
&&\left\vert A_{\lambda }\left( J_{\lambda }^{\mu, \varepsilon }\left( X_{\lambda
}^{\varepsilon }\right) \right) \right\vert _{V^{\prime }}^{2}
=_{V^{\prime }}\left\langle A_{\lambda }\left( J_{\lambda }^{\mu, \varepsilon
}\left( X_{\lambda }^{\varepsilon }\right) \right) ,\varphi \right\rangle
_{V}\medskip \\
&=&\int_{\mathcal{O}}\left( \nabla \tilde{\Psi}_\lambda \pr{J_{\lambda }^{\mu, \varepsilon }\left( X_{\lambda
}^{\varepsilon }\right)} \cdot \nabla \varphi
-KJ_{\lambda }^{\mu, \varepsilon }\left( X_{\lambda
}^{\varepsilon }\right){\frac{\partial \varphi}{\partial x_{3}}}\right) d\xi 
+\int_{\Gamma _{u}}\alpha Tr\pr{\tilde{\Psi}_\lambda \pr{J_{\lambda }^{\mu, \varepsilon }\left( X_{\lambda
}^{\varepsilon }\right)}}
Tr\left( \varphi \right) d\sigma,
\end{eqnarray*}%
where $\varphi $ satisfies 
\begin{equation*}
-\Delta \varphi =A_{\lambda }\left( J_{\lambda }^{\mu, \varepsilon }\left(
X_{\lambda }^{\varepsilon }\right) \right) ,\textnormal{ on } \mathcal{O}, \ 
\dfrac{\partial \varphi }{\partial \nu }+\alpha \varphi =0, \textnormal{ on }\Gamma _{u},\ 
\dfrac{\partial \varphi }{\partial \nu }=0, \textnormal{ on }\Gamma _{s}.%
\end{equation*}

\noindent Green's formula gives
\begin{equation}\label{est1}
\left\vert A_{\lambda }\left( J_{\lambda }^{\mu, \varepsilon }\left( X_{\lambda
}^{\varepsilon }\right) \right) \right\vert _{V^{\prime }}^{2}\medskip
=\int_{\mathcal{O}}\widetilde{\Psi }_{\lambda }\left( J_{\lambda
}^{\mu, \varepsilon }\left( X_{\lambda }^{\varepsilon }\right) \right) A_{\lambda
}\left( J_{\lambda }^{\mu, \varepsilon }\left( X_{\lambda }^{\varepsilon }\right)
\right) d\xi -\int_{\mathcal{O}}KJ_{\lambda }^{\mu, \varepsilon }\left(
X_{\lambda }^{\varepsilon }\right) \frac{\partial \varphi }{\partial x_{3}}%
d\xi 
{=:}I_{1}+I_{2}. 
\end{equation}

By using the form of the operator $A_{\lambda }$, and the fact that
the $Tr$ function is bounded by the $H^{1}\left( \mathcal{O}\right) $ norm,
it follows that
\begin{eqnarray*}
I_{1} &=&\int_{\mathcal{O}}\widetilde{\Psi }_{\lambda }\left( J_{\lambda
}^{\mu,\varepsilon }\left( X_{\lambda }^{\varepsilon }\right) \right) A_{\lambda
}\left( J_{\lambda }^{\mu, \varepsilon }\left( X_{\lambda }^{\varepsilon }\right)
\right) d\xi \\
&=&\int_{\mathcal{O}}\left\vert \nabla \widetilde{\Psi }_{\lambda }\left(
J_{\lambda }^{\mu, \varepsilon }\left( X_{\lambda }^{\varepsilon }\right) \right)
\right\vert ^{2}d\xi -\int_{\mathcal{O}}KJ_{\lambda }^{\mu, \varepsilon }\left(
X_{\lambda }^{\varepsilon }\right) \frac{\partial }{\partial x_{3}}%
\widetilde{\Psi }_{\lambda }\left( J_{\lambda }^{\mu, \varepsilon }\left(
X_{\lambda }^{\varepsilon }\right) \right) d\xi  \\&&+\int_{\Gamma
_{u}}\alpha \left\vert Tr\left( \widetilde{\Psi }_{\lambda }\left(
J_{\lambda }^{\mu, \varepsilon }\left( X_{\lambda }^{\varepsilon }\right) \right)
\right) \right\vert ^{2}d\sigma \\
&\leq &C\int_{\mathcal{O}}\left\vert \nabla \widetilde{\Psi }_{\lambda
}\left( J_{\lambda }^{\mu, \varepsilon }\left( X_{\lambda }^{\varepsilon }\right)
\right) \right\vert ^{2}d\xi +C\int_{\mathcal{O}}\left\vert J_{\lambda
}^{\mu,\varepsilon }\left( X_{\lambda }^{\varepsilon }\right) \right\vert
^{2}d\xi +\alpha
_{M}\left\vert \widetilde{\Psi }_{\lambda }\left( J_{\lambda }^{\mu,\varepsilon
}\left( X_{\lambda }^{\varepsilon }\right) \right) \right\vert _{H^{1}\left( 
\mathcal{O}\right) }^{2}.
\end{eqnarray*}

\noindent Since $\widetilde{\Psi }_{\lambda }$ is assumed to be Lipschitz 
\begin{eqnarray*}
&&\int_{\mathcal{O}}\widetilde{\Psi }_{\lambda }\left( J_{\lambda
}^{\mu,\varepsilon }\left( X_{\lambda }^{\varepsilon }\right) \right) A_{\lambda
}\left( J_{\lambda }^{\mu, \varepsilon }\left( X_{\lambda }^{\varepsilon }\right)
\right) d\xi \leq C\int_{\mathcal{O}}\left\vert \nabla \widetilde{\Psi }_{\lambda
}\left( J_{\lambda }^{\mu, \varepsilon }\left( X_{\lambda }^{\varepsilon }\right)
\right) \right\vert ^{2}d\xi +C\left( \lambda \right) \int_{\mathcal{O}%
}\left\vert J_{\lambda }^{\mu, \varepsilon }\left( X_{\lambda }^{\varepsilon
}\right) \right\vert ^{2}d\xi.
\end{eqnarray*}

\noindent On the other hand, using arguments in the spirit of \eqref{eess},
\begin{eqnarray*}
I_{2} &=&-K\int_{\mathcal{O}}J_{\lambda }^{\mu, \varepsilon }\left( X_{\lambda
}^{\varepsilon }\right) \frac{\partial \varphi }{\partial x_{3}}d\xi\leq C\int_{\mathcal{O}}\left\vert J_{\lambda }^{\mu, \varepsilon }\left(
X_{\lambda }^{\varepsilon }\right) \right\vert ^{2}d\xi +\frac{1}{2}%
\left\vert A_{\lambda }\left( J_{\lambda }^{\mu, \varepsilon }\left( X_{\lambda
}^{\varepsilon }\right) \right) \right\vert _{V^{\prime }}^{2}.
\end{eqnarray*}

By replacing $I_{1}$ and $I_{2}$ in (\ref{est1}), we get%
\begin{equation}
\frac{1}{2}\left\vert A_{\lambda }\left( J_{\lambda }^{\mu, \varepsilon }\left(
X_{\lambda }^{\varepsilon }\right) \right) \right\vert _{V^{\prime
}}^{2}\leq C\int_{\mathcal{O}}\left\vert \nabla \widetilde{\Psi }_{\lambda
}\left( J_{\lambda }^{\mu, \varepsilon }\left( X_{\lambda }^{\varepsilon }\right)
\right) \right\vert ^{2}d\xi +C\left( \lambda \right) \int_{\mathcal{O}%
}\left\vert J_{\lambda }^{\mu, \varepsilon }\left( X_{\lambda }^{\varepsilon
}\right) \right\vert ^{2}d\xi.  \label{estt}
\end{equation}

Since $\mathbb{E}\int_{0}^{t}\int_{\mathcal{O}}\left\vert J_{\lambda
}^{\mu,\varepsilon }\left( X_{\lambda }^{\varepsilon }\right) \right\vert
^{2}d\xi $ is bounded uniformly in $\varepsilon ,$ it is now sufficient to
bound the first term on the right-hand of \eqref{estt}.

First, the reader is invited to note that
\begin{equation}\label{calc5}\begin{split}
&\int_{\mathcal{O}}\nabla \widetilde{\Psi }_{\lambda }\left(
J_{\lambda }^{\mu, \varepsilon }\left( X_{\lambda }^{\varepsilon }\right) \right)\ \cdot\ \nabla J_{\lambda }^{\mu, \varepsilon }\left( X_{\lambda }^{\varepsilon }\right)
d\xi =\int_{\mathcal{O}}\nabla \widetilde{\Psi }_{\lambda }\left(
J_{\lambda }^{\mu, \varepsilon }\left( X_{\lambda }^{\varepsilon }\right) \right) \ \cdot\ \nabla \widetilde{\Psi }_{\lambda }^{-1}\left( \widetilde{\Psi }_{\lambda
}\left( J_{\lambda }^{\mu, \varepsilon }\left( X_{\lambda }^{\varepsilon }\right)
\right) \right) d\xi \\
&=\int_{\mathcal{O}}\left\vert \nabla \widetilde{\Psi }_{\lambda }\left(
J_{\lambda }^{\mu, \varepsilon }\left( X_{\lambda }^{\varepsilon }\right) \right)
\right\vert ^{2}\left( \widetilde{\Psi }_{\lambda }^{-1}\right) ^{\prime
}\pr{J_{\lambda }^{\mu, \varepsilon }\left( X_{\lambda }^{\varepsilon }\right) }d\xi \geq c\int_{\mathcal{O}}\left\vert \nabla \widetilde{\Psi }_{\lambda
}\left( J_{\lambda }^{\mu, \varepsilon }\left( X_{\lambda }^{\varepsilon }\right)
\right) \right\vert ^{2}d\xi ,
\end{split}\end{equation}
because $\tilde{\Psi}_{\lambda }^{-1}$ is strictly increasing which
follows from the fact that $\widetilde{\Psi }_{\lambda }$ is Lipschitz-continuous. The constant $c>0$ is generic, but we write it as a lower case since we will be employing it hereafter.\\
\noindent The definition of $A_\lambda$ yields
\begin{eqnarray*}
&&\int_{\mathcal{O}}\nabla \widetilde{\Psi }_{\lambda }\left(
J_{\lambda }^{\mu, \varepsilon }\left( X_{\lambda }^{\varepsilon }\right) \right)
\cdot \nabla J_{\lambda }^{\mu, \varepsilon }\left( X_{\lambda }^{\varepsilon }\right)
d\xi\\ &\leq& _{V^{\prime }}\left\langle A_{\lambda }\left( J_{\lambda
}^{\mu,\varepsilon }\left( X_{\lambda }^{\varepsilon }\right) \right), J_{\lambda }^{\mu, \varepsilon }\left( X_{\lambda }^{\varepsilon }\right)
\right\rangle _{V}+K\int_{\mathcal{O}}J_{\lambda }^{\mu, \varepsilon }\left(
X_{\lambda }^{\varepsilon }\right) \frac{\partial }{\partial x_{3}}%
J_{\lambda }^{\mu, \varepsilon }\left( X_{\lambda }^{\varepsilon }\right) d\xi\\
&&-\int_{\Gamma _{u}}\alpha Tr\pr{\widetilde{\Psi }_{\lambda }\left( J_{\lambda
}^{\mu,\varepsilon }\left( X_{\lambda }^{\varepsilon }\right) \right) }Tr\pr{J_{\lambda
}^{\mu,\varepsilon }\left( X_{\lambda }^{\varepsilon }\right)} d\sigma\\
&\leq& _{V^{\prime }}\left\langle A_{\lambda }\left( J_{\lambda
}^{\mu,\varepsilon }\left( X_{\lambda }^{\varepsilon }\right) \right)
,J_{\lambda }^{\mu, \varepsilon }\left( X_{\lambda }^{\varepsilon }\right)
\right\rangle _{V}+K\int_{\mathcal{O}}J_{\lambda }^{\mu, \varepsilon }\left(
X_{\lambda }^{\varepsilon }\right) \frac{\partial }{\partial x_{3}}%
J_{\lambda }^{\mu, \varepsilon }\left( X_{\lambda }^{\varepsilon }\right) d\xi.
\end{eqnarray*}
The latter inequality follows from the fact that 
\[Tr\pr{\widetilde{\Psi }_{\lambda }\left( J_{\lambda
}^{\mu,\varepsilon }\left( X_{\lambda }^{\varepsilon }\right) \right) }Tr\pr{J_{\lambda
}^{\mu,\varepsilon }\left( X_{\lambda }^{\varepsilon }\right)}=Tr\pr{\widetilde{\Psi }_{\lambda }\left( J_{\lambda
}^{\mu,\varepsilon }\left( X_{\lambda }^{\varepsilon }\right) \right)\pr{J_{\lambda
}^{\mu,\varepsilon }\left( X_{\lambda }^{\varepsilon }\right)}},\]where the right-hand operator acts on $W^{1,1}$, and its argument is non-negative due to the monotonicity of $\widetilde{\Psi }_{\lambda }$.
The reader is reminded that $\tilde{\Psi}_\lambda^{-1}$ is $\frac{1}{\lambda}$-Lipschitz continuous. As a consequence, for $\delta>0$, one has 
\[\abs{J_\lambda^{\mu,\varepsilon}(x)\frac{\partial }{\partial x_3}J_\lambda^{\mu,\varepsilon(x)}}\leq \delta\abs{\nabla J_\lambda^{\mu,\varepsilon}(x)}^2+\frac{1}{4\delta}\abs{J_\lambda^{\mu,\varepsilon}(x)}^2\leq C(\lambda)\delta\abs{\nabla \tilde{\Psi}_\lambda\pr{J_\lambda^{\mu,\varepsilon}(x)}}^2+\frac{1}{4\delta}\abs{J_\lambda^{\mu,\varepsilon}(x)}^2.\] Owing to \eqref{calc5}, by picking $\delta$ small enough,
\begin{equation}\label{estim+}
\begin{split}
c\int_{\mathcal{O}}\left\vert \nabla \widetilde{\Psi }_{\lambda
}\left( J_{\lambda }^{\mu, \varepsilon }\left( X_{\lambda }^{\varepsilon }\right)
\right) \right\vert ^{2}d\xi\leq &_{V^{\prime }}\left\langle A_{\lambda }\left( J_{\lambda
}^{\mu,\varepsilon }\left( X_{\lambda }^{\varepsilon }\right) \right)
,J_{\lambda }^{\mu, \varepsilon }\left( X_{\lambda }^{\varepsilon }\right)
\right\rangle _{V}+C\left( \lambda \right) \int_{\mathcal{O}}\left\vert
J_{\lambda }^{\mu, \varepsilon }\left( X_{\lambda }^{\varepsilon }\right)
\right\vert ^{2}d\xi \\
 &+\frac{c}{2}\int_{\mathcal{O}}\left\vert \nabla \widetilde{%
\Psi }_{\lambda }\left( J_{\lambda }^{\mu, \varepsilon }\left( X_{\lambda
}^{\varepsilon }\right) \right) \right\vert ^{2}d\xi .\end{split}
\end{equation}
We recall the estimate (\ref{Ldoi}) in order to obtain, from \eqref{estim+} and \eqref{estt}, the estimates \eqref{A}, i.e.,
\begin{equation*}
\mathbb{E}\int_{0}^{t}\left\vert A_{\lambda }^{\mu ,\varepsilon }\left(
X_{\lambda }^{\varepsilon }\left( r\right) \right) \right\vert _{V^{\prime
}}^{2}dr\leq C\left( \lambda \right) .
\end{equation*}
\textbf{Step 4. (strong convergences)}
In order to conclude the proof of the Lemma we shall now prove that%
\begin{equation}
X_{\lambda }^{\varepsilon }\rightarrow X_{\lambda }\text{ strongly in }%
L^{\infty }\left( 0,T;L^{2}\left( \Omega ;V^{\prime }\right) \right) .
\label{strong1}
\end{equation}%
and 
\begin{equation}
\Psi _{\lambda }\left( J_{\lambda }^{\mu, \varepsilon }\left( X_{\lambda
}^{\varepsilon }\right) \right) \rightarrow \eta \text{ strongly in }%
L^{2}\left( 0,T;L^{2}\left( \Omega ;L^{2}\left( \mathcal{O}\right) \right)
\right) .  \label{psi}
\end{equation}

To this purpose, we apply It\^{o}'s formula with the $V^{\prime }$ squared norm
to $X_{\lambda }^{\varepsilon }-X_{\lambda }^{\varepsilon ^{\prime }}$ on $\pp{0,t}$.
\begin{eqnarray*}
&&\mathbb{E}\left\vert X_{\lambda }^{\varepsilon }(t)-X_{\lambda }^{\varepsilon
^{\prime }}(t)\right\vert _{V^{\prime }}^{2} +2\mathbb{E}\int_{0}^{t}\left\langle A_{\lambda }^{\mu ,\varepsilon
}\left( X_{\lambda }^{\varepsilon }\right) -A_{\lambda }^{\mu ,\varepsilon
^{\prime }}\left( X_{\lambda }^{\varepsilon ^{\prime }}\right) ,X_{\lambda
}^{\varepsilon }-X_{\lambda }^{\varepsilon ^{\prime }}\right\rangle
_{V^{\prime }}dr \\
&&-2\mu \mathbb{E}\int_{0}^{t}\left\vert X_{\lambda }^{\varepsilon
}-X_{\lambda }^{\varepsilon ^{\prime }}\right\vert _{V^{\prime }}^{2}dr+2%
\mathbb{E}\int_{0}^{t}\left\langle F_{u}+F_{s},X_{\lambda }^{\varepsilon
}-X_{\lambda }^{\varepsilon ^{\prime }}\right\rangle _{V^{\prime }}dr \leq C\mathbb{E}\int_{0}^{t}\left\vert X_{\lambda }^{\varepsilon
}-X_{\lambda }^{\varepsilon ^{\prime }}\right\vert _{V^{\prime }}^{2}dr.
\end{eqnarray*}

\noindent For the term involving the differential operators $A_\lambda^{\mu,\cdot}$, one has
\begin{equation}\label{estim++}
\begin{split}
&\left\langle A_{\lambda }^{\mu ,\varepsilon }\left( X_{\lambda
}^{\varepsilon }\right) -A_{\lambda }^{\mu ,\varepsilon ^{\prime }}\left(
X_{\lambda }^{\varepsilon ^{\prime }}\right) ,X_{\lambda }^{\varepsilon
}-X_{\lambda }^{\varepsilon ^{\prime }}\right\rangle _{V^{\prime }} \\
=&\left\langle A_{\lambda }^{\mu }\left( J_{\lambda }^{\mu, \varepsilon }\left(
X_{\lambda }^{\varepsilon }\right) \right) -A_{\lambda }^{\mu }\left(
J_{\lambda }^{\mu,\varepsilon' }\left( X_{\lambda }^{\varepsilon ^{\prime
}}\right) \right) ,J_{\lambda }^{\mu, \varepsilon }\left( X_{\lambda
}^{\varepsilon }\right) -J_{\lambda }^{\mu,\varepsilon' }\left( X_{\lambda
}^{\varepsilon ^{\prime }}\right) \right\rangle _{V^{\prime }} \\
&+\left\langle A_{\lambda }^{\mu }\left( J_{\lambda }^{\mu, \varepsilon }\left(
X_{\lambda }^{\varepsilon }\right) \right) -A_{\lambda }^{\mu }\left(
J_{\lambda }^{\mu,\varepsilon' }\left( X_{\lambda }^{\varepsilon ^{\prime
}}\right) \right) ,\varepsilon A_{\lambda }^{\mu }\left( J_{\lambda
}^{\mu,\varepsilon }\left( X_{\lambda }^{\varepsilon }\right) \right)
-\varepsilon ^{\prime }A_{\lambda }^{\mu }\left( J_{\lambda }^{\mu,\varepsilon'
}\left( X_{\lambda }^{\varepsilon ^{\prime }}\right) \right) \right\rangle
_{V^{\prime }}\\
\geq &\mu \left\vert J_{\lambda }^{\mu, \varepsilon }\left( X_{\lambda
}^{\varepsilon }\right) -J_{\lambda }^{\mu,\varepsilon ^{\prime }}\left(
X_{\lambda }^{\varepsilon ^{\prime }}\right) \right\vert _{V^{\prime }}^{2}
+\left\langle \widetilde{\Psi }_{\lambda }\left( J_{\lambda }^{\mu,\varepsilon
}\left( X_{\lambda }^{\varepsilon }\right) \right) -\widetilde{\Psi }%
_{\lambda }\left( J_{\lambda }^{\mu,\varepsilon ^{\prime }}\left( X_{\lambda
}^{\varepsilon ^{\prime }}\right) \right) ,J_{\lambda }^{\mu, \varepsilon }\left(
X_{\lambda }^{\varepsilon }\right) -J_{\lambda }^{\mu, \varepsilon }\left(
X_{\lambda }^{\varepsilon ^{\prime }}\right) \right\rangle _{2} \\
&-K\left\langle J_{\lambda }^{\mu, \varepsilon }\left( X_{\lambda }^{\varepsilon
}\right) -J_{\lambda }^{\mu,\varepsilon ^{\prime }}\left( X_{\lambda
}^{\varepsilon ^{\prime }}\right) ,\frac{\partial \varphi }{\partial x_{3}}%
\right\rangle _{2} -2\left( \varepsilon +\varepsilon ^{\prime }\right) \left( \left\vert
A_{\lambda }^{\mu ,\varepsilon }\left( X_{\lambda }^{\varepsilon }\right)
\right\vert _{V^{\prime }}^{2}+\left\vert A_{\lambda }^{\mu ,\varepsilon
}\left( X_{\lambda }^{\varepsilon ^{\prime }}\right) \right\vert _{V^{\prime
}}^{2}\right)
\end{split}
\end{equation}%
where $\varphi $ satisfies  
\begin{equation*}
-\Delta \varphi =J_{\lambda }^{\mu, \varepsilon }\left( X_{\lambda }^{\varepsilon
}\right) -J_{\lambda }^{\mu, \varepsilon }\left( X_{\lambda }^{\varepsilon
^{\prime }}\right) , \textnormal{ on } \mathcal{O},\ 
\dfrac{\partial \varphi }{\partial \nu }+\alpha \varphi =0, \textnormal{ on }\Gamma _{u},\ 
\dfrac{\partial \varphi }{\partial \nu }=0, \textnormal{ on }\Gamma _{s}.%
\end{equation*}
Going back to \eqref{estim++}, one gets
\begin{eqnarray*}
&&\mathbb{E}\left\vert X_{\lambda }^{\varepsilon }-X_{\lambda }^{\varepsilon
^{\prime }}\right\vert _{V^{\prime }}^{2}+2\mu \mathbb{E}\int_{0}^{t}\left\vert J_{\lambda }^{\mu, \varepsilon }\left(
X_{\lambda }^{\varepsilon }\right) -J_{\lambda }^{\mu,\varepsilon ^{\prime
}}\left( X_{\lambda }^{\varepsilon ^{\prime }}\right) \right\vert
_{V^{\prime }}^{2}dr +2\lambda \mathbb{E}\int_{0}^{t}\left\vert J_{\lambda }^{\mu,\varepsilon
}\left( X_{\lambda }^{\varepsilon }\right) -J_{\lambda }^{\mu,\varepsilon
^{\prime }}\left( X_{\lambda }^{\varepsilon ^{\prime }}\right) \right\vert
_{2}^{2}dr \\
&&+2\mathbb{E}\int_{0}^{t}\left\langle \Psi _{\lambda }\left( J_{\lambda
}^{\mu,\varepsilon }\left( X_{\lambda }^{\varepsilon }\right) \right) -\Psi
_{\lambda }\left( J_{\lambda }^{\mu,\varepsilon ^{\prime }}\left( X_{\lambda
}^{\varepsilon ^{\prime }}\right) \right) ,J_{\lambda }^{\mu, \varepsilon }\left(
X_{\lambda }^{\varepsilon }\right) -J_{\lambda }^{\mu,\varepsilon '}\left(
X_{\lambda }^{\varepsilon ^{\prime }}\right) \right\rangle _{2}dr \\
&&-2K\mathbb{E}\int_{0}^{t}\left\langle J_{\lambda }^{\mu, \varepsilon }\left(
X_{\lambda }^{\varepsilon }\right) -J_{\lambda }^{\mu,\varepsilon ^{\prime
}}\left( X_{\lambda }^{\varepsilon ^{\prime }}\right) ,\frac{\partial
\varphi }{\partial x_{3}}\right\rangle _{2}dr \\&&-4\left( \varepsilon +\varepsilon ^{\prime }\right) \mathbb{E}%
\int_{0}^{t}\left( \left\vert A_{\lambda }^{\mu ,\varepsilon }\left(
X_{\lambda }^{\varepsilon }\right) \right\vert _{V^{\prime }}^{2}+\left\vert
A_{\lambda }^{\mu ,\varepsilon '}\left( X_{\lambda }^{\varepsilon ^{\prime
}}\right) \right\vert _{V^{\prime }}^{2}\right) dr\leq C\mathbb{E}\int_{0}^{t}\left\vert X_{\lambda }^{\varepsilon
}-X_{\lambda }^{\varepsilon ^{\prime }}\right\vert _{V^{\prime }}^{2}dr.
\end{eqnarray*}
Since
\begin{eqnarray*}
&&-2K\left\langle J_{\lambda }^{\mu, \varepsilon }\left( X_{\lambda
}^{\varepsilon }\right) -J_{\lambda }^{\mu,\varepsilon ^{\prime }}\left(
X_{\lambda }^{\varepsilon ^{\prime }}\right) ,\frac{\partial \varphi }{%
\partial x_{3}}\right\rangle _{2}\medskip  \\&\geq& -\lambda \left\vert J_{\lambda }^{\mu, \varepsilon }\left( X_{\lambda
}^{\varepsilon }\right) -J_{\lambda }^{\mu,\varepsilon ^{\prime }}\left(
X_{\lambda }^{\varepsilon ^{\prime }}\right) \right\vert _{2}^{2}-\frac{K^{2}%
}{\lambda }\left\vert J_{\lambda }^{\mu, \varepsilon }\left( X_{\lambda
}^{\varepsilon }\right) -J_{\lambda }^{\mu,\varepsilon ^{\prime }}\left(
X_{\lambda }^{\varepsilon ^{\prime }}\right) \right\vert _{V^{\prime }}^{2},
\end{eqnarray*}%
by the same argument as in the Lemma \ref{lema1} (see \eqref{eess}), and since the strong monotonicity of $\Psi _{\lambda }^{-1}$ yields, for some generic $c(\lambda)>0$,%
\begin{eqnarray*}
&&\left\langle \Psi _{\lambda }\left( J_{\lambda }^{\mu, \varepsilon }\left(
X_{\lambda }^{\varepsilon }\right) \right) -\Psi _{\lambda }\left(
J_{\lambda }^{\mu,\varepsilon ^{\prime }}\left( X_{\lambda }^{\varepsilon
^{\prime }}\right) \right) ,J_{\lambda }^{\mu, \varepsilon }\left( X_{\lambda
}^{\varepsilon }\right) -J_{\lambda }^{\mu, \varepsilon }\left( X_{\lambda
}^{\varepsilon ^{\prime }}\right) \right\rangle _{2} \\
&\geq &c\left( \lambda \right) \left\vert \Psi _{\lambda }\left( J_{\lambda
}^{\mu,\varepsilon }\left( X_{\lambda }^{\varepsilon }\right) \right) -\Psi
_{\lambda }\left( J_{\lambda }^{\mu,\varepsilon ^{\prime }}\left( X_{\lambda
}^{\varepsilon ^{\prime }}\right) \right) \right\vert _{2}^{2},
\end{eqnarray*}%
one gets
\begin{eqnarray*}
&&\mathbb{E}\left\vert X_{\lambda }^{\varepsilon }-X_{\lambda }^{\varepsilon
^{\prime }}\right\vert _{V^{\prime }}^{2}+\pr{2\mu -\frac{K^2}{\lambda }} \mathbb{E}\int_{0}^{t}\left\vert J_{\lambda }^{\mu, \varepsilon }\left(
X_{\lambda }^{\varepsilon }\right) -J_{\lambda }^{\mu,\varepsilon ^{\prime
}}\left( X_{\lambda }^{\varepsilon ^{\prime }}\right) \right\vert
_{V^{\prime }}^{2}dr \\&&+\lambda \mathbb{E}\int_{0}^{t}\left\vert J_{\lambda }^{\mu,\varepsilon
}\left( X_{\lambda }^{\varepsilon }\right) -J_{\lambda }^{\mu,\varepsilon
^{\prime }}\left( X_{\lambda }^{\varepsilon ^{\prime }}\right) \right\vert
_{2}^{2}dr +2c(\lambda)\mathbb{E}\int_{0}^{t}\left\vert \Psi _{\lambda }\left( J_{\lambda
}^{\mu,\varepsilon }\left( X_{\lambda }^{\varepsilon }\right) \right) -\Psi
_{\lambda }\left( J_{\lambda }^{\mu,\varepsilon ^{\prime }}\left( X_{\lambda
}^{\varepsilon ^{\prime }}\right) \right) \right\vert _{2}^{2}dr \\
&&-4\left( \varepsilon +\varepsilon ^{\prime }\right)\mathbb{E}%
\int_{0}^{t}\left( \left\vert A_{\lambda }^{\mu ,\varepsilon }\left(
X_{\lambda }^{\varepsilon }\right) \right\vert _{V^{\prime }}^{2}+\left\vert
A_{\lambda }^{\mu ,\varepsilon }\left( X_{\lambda }^{\varepsilon ^{\prime
}}\right) \right\vert _{V^{\prime }}^{2}\right) dr \leq C\mathbb{E}\int_{0}^{t}\left\vert X_{\lambda }^{\varepsilon
}-X_{\lambda }^{\varepsilon ^{\prime }}\right\vert _{V^{\prime }}^{2}dr.
\end{eqnarray*}

Finally, by Gronwall's inequality and keeping in mind (\ref{A}),
for $\mu$ sufficiently large (larger than $\frac{K^2}{2\lambda}$), one gets%
\begin{equation}\label{est7}
\begin{split}
&\mathbb{E}\left\vert X_{\lambda }^{\varepsilon }-X_{\lambda }^{\varepsilon
^{\prime }}\right\vert _{V^{\prime }}^{2}+\mathbb{E}\int_{0}^{t}\left\vert J_{\lambda }^{\mu, \varepsilon }\left(
X_{\lambda }^{\varepsilon }\right) -J_{\lambda }^{\mu,\varepsilon ^{\prime
}}\left( X_{\lambda }^{\varepsilon ^{\prime }}\right) \right\vert
_{V^{\prime }}^{2}dr+\mathbb{E}\int_{0}^{t}\left\vert J_{\lambda }^{\mu,\varepsilon
}\left( X_{\lambda }^{\varepsilon }\right) -J_{\lambda }^{\mu,\varepsilon
^{\prime }}\left( X_{\lambda }^{\varepsilon ^{\prime }}\right) \right\vert
_{2}^{2}dr \\&+\mathbb{E}\int_{0}^{t}\left\vert \Psi _{\lambda }\left( J_{\lambda
}^{\mu,\varepsilon }\left( X_{\lambda }^{\varepsilon }\right) \right) -\Psi
_{\lambda }\left( J_{\lambda }^{\mu,\varepsilon ^{\prime }}\left( X_{\lambda
}^{\varepsilon ^{\prime }}\right) \right) \right\vert _{2}^{2}dr \leq C\left( \lambda \right) \left( \varepsilon +\varepsilon ^{\prime
}\right) .  \end{split}
\end{equation}
The proof of the lemma is complete by letting $\varepsilon ,\varepsilon
^{\prime }\rightarrow 0$.
\end{proof}
\bigskip

\section{Appendix}\label{App}
\subsection{Estimates on the eigenfunctions and well-posedness of $\Sigma$}
Let us consider the Laplace operator with Robin boundary conditions on an open
bounded domain $\mathcal{O}\in \mathbb{R}^{d}$ with regular boundary.\ We consider the
eigenfunctions $\left\{ e_{j}\right\} _{j}$ and the eigenvalues $\left\{
\lambda _{j}\right\} _{j}$,\ that is, we look into the following problem%
\begin{equation*}
\left\{ 
\begin{array}{ll}
-\Delta e_{j}=\lambda _{j}e_{j}, &\textnormal{ on } \mathcal{O}, \\ 
\dfrac{\partial e_{j}}{\partial \nu }+\alpha e_{j}=0, &\textnormal{ on } \Gamma _{u}, \\ 
\dfrac{\partial e_{j}}{\partial \nu }=0, &\textnormal{ on } \Gamma _{s}.%
\end{array}%
\right.
\end{equation*}%
where $\partial \mathcal{O}=\Gamma $ is sufficiently smooth, formed by the
disjoint parts $\Gamma _{u}$ and $\Gamma _{s},$ i.e.%
\begin{equation*}
\Gamma =\overline{\Gamma }_{u}\cup \Gamma _{s},\quad \Gamma _{u}\cap \Gamma
_{s}=\emptyset .
\end{equation*}
%This problem is well posed by classical theory (see e.g. \cite{limita}).
This problem has a sequence of solutions $(\lambda_j, \varphi_j)_{j \in \mathbb{N}}$ with $\lambda_j \geq 0$ a growing sequence which tends to infinity. This result can be easily deduced using classical results on self-adjoint operators with compact resolvant, see for instance Theorem 3.10, page 53 in \cite{bonnet1995theorie}.

\begin{proposition}
There exist two real constants $\overline{C}>0$
and $\widetilde{C}>0$ such that, for every $x\in L^{2}\left( \mathcal{O}\right) $, 
\begin{equation*}
\left\vert xe_{j}\right\vert _{2}^{2}\leq \widetilde{C}\lambda
_{j}^{\frac{d-1}{2}}\left\vert x\right\vert _{2}^{2},\quad \forall j\in \mathbb{N},
\end{equation*}%
and%
\begin{equation*}
\left\vert xe_{j}\right\vert _{V^{\prime }}^{2}\leq \overline{C}\pr{1+\lambda
_{j}^{\frac{d+1}{2}}}\left\vert x\right\vert _{V^{\prime }}^{2},\quad \forall j\in 
\mathbb{N}.
\end{equation*}
\end{proposition}

\begin{proof}
To prove the first assertion, one notes that
\begin{equation*}
\left\vert xe_{j}\right\vert _{2}\leq \left\vert e_{j}\right\vert _{\infty
}\left\vert x\right\vert _{2}\leq C\lambda _{j}^{\frac{d-1}{4}}\left\vert
x\right\vert _{2},\quad \forall j\in \mathbb{N}\text{.}
\end{equation*}%
Indeed, in the spirit of of \cite[Theorem 1]{66}, we have that $\left\vert e_{j}\right\vert _{\infty
}\leq \widetilde{C}\lambda _{j}^{\frac{d-1}{4}}$ for all $k\in \mathbb{N}$.
The reader is invited to note that we consider the eigenvalues to be $\lambda_j$, while \cite{66} deals with $\lambda_j^2$ as eigenvalue. The constant $C$ is generic and depends on the domain $\mathcal{O}$, but not on $j$.\\
In order to prove the second inequality, we consider the functional
framework which was introduced at the beginning of the paper. One has 
\begin{equation*}
\left\vert xe_{j}\right\vert _{V^{\prime }}^{2}=\left\langle
xe_{j},xe_{j}\right\rangle _{V^{\prime }}=(xe_{j})\left( \varphi \right),
\end{equation*}%
where $\varphi $\ is the solution of%
\begin{equation*}
-\Delta \varphi =e_{j}x, \textnormal{ on } \mathcal{O},\ 
\dfrac{\partial \varphi }{\partial \nu }+\alpha \varphi =0, \textnormal{ on } \Gamma _{u}, \ 
\dfrac{\partial \varphi }{\partial \nu }=0, \textnormal{ on } \Gamma _{s}.%
\end{equation*}

Since $x\in L^{2}\left( \mathcal{O}\right), $ it follows that
\begin{equation}
\abs{(e_{j}x)\left( \varphi \right) }=\abs{\left\langle e_{j}x,\varphi \right\rangle
_{2}}=\abs{\int\limits_{\mathcal{O}}e_{j}\left( \xi \right) x\left( \xi \right)
\varphi \left( \xi \right) d\xi }\leq \left\vert x\right\vert _{V^{\prime
}}\left\vert e_{j}\varphi \right\vert _{V}  \label{star}
\end{equation}%
by the Gelfand triple $V\subset L^{2}\left( \mathcal{O}\right) \subset
V^{\prime }$.
One computes
\begin{eqnarray*}
\left\vert e_{j}\varphi \right\vert _{V}^{2} &=&\int\limits_{\mathcal{O}%
}\left\vert \nabla \left( e_{j}\varphi \right) \right\vert ^{2}d\xi
+\int\limits_{\Gamma_u}\alpha \left\vert e_{j}\varphi
\right\vert ^{2}d\sigma \\
&=&-\int\limits_{\mathcal{O}}\Delta \left( e_{j}\varphi \right) \left(
e_{j}\varphi \right) d\xi +\int\limits_{\partial \mathcal{O}}\frac{\partial
\left( e_{j}\varphi \right) }{\partial \nu }\left( e_{j}\varphi \right)
d\sigma +\int\limits_{{\Gamma_u}}\alpha \left\vert e_{j}\varphi
\right\vert ^{2}d\sigma .
\end{eqnarray*}
The reader is invited to note that, by a slight abuse of notations, we have dropped the $Tr$ operator, but it should still be kept on all the elements integrated w.r.t $d\sigma$. For the first term , one gets%
\begin{eqnarray*}
-\int\limits_{\mathcal{O}}\Delta \left( e_{j}\varphi \right) \left(
e_{j}\varphi \right) d\xi
&=&-\int\limits_{\mathcal{O}}\left( e_{j}\varphi ^{2}\Delta
e_{j}+e_{j}^{2}\varphi \Delta \varphi +\frac{1}{2}\nabla \left(
e_{j}^{2}\right)\ \cdot \ \nabla \left( \varphi ^{2}\right) \right) d\xi \\
&=&-\int\limits_{\mathcal{O}}\left( e_{j}\varphi ^{2}\Delta
e_{j}+e_{j}^{2}\varphi \Delta \varphi \right) d\xi +\frac{1}{2}\int\limits_{%
\mathcal{O}}e_{j}^{2}\Delta \left( \varphi ^{2}\right) d\xi -\frac{1}{2}%
\int\limits_{\partial \mathcal{O}}e_{j}^{2}\frac{\partial \varphi ^{2}}{\partial
\nu }d\sigma .
\end{eqnarray*}
Going back to the initial expression yields
\begin{eqnarray*}
\left\vert e_{j}\varphi \right\vert _{V}^{2} &=&-\int\limits_{\mathcal{O}%
}\left( e_{j}\varphi ^{2}\Delta e_{j}+e_{j}^{2}\varphi \Delta \varphi
\right) d\xi +\frac{1}{2}\int\limits_{\mathcal{O}}e_{j}^{2}\Delta \left(
\varphi ^{2}\right) d\xi \\
&&-\frac{1}{2}\int\limits_{\partial \mathcal{O}}e_{j}^{2}\frac{\partial \varphi ^{2}%
}{\partial \nu }d\sigma +\int\limits_{\partial \mathcal{O}}\frac{%
\partial \left( e_{j}\varphi \right) }{\partial \nu }\left( e_{j}\varphi
\right) d\sigma +\int\limits_{\Gamma_u}\alpha \left\vert
e_{j}\varphi \right\vert ^{2}d\sigma \\
&=&\int\limits_{\mathcal{O}}\lambda _{j}e_{j}^{2}\varphi ^{2}d\xi
-\int\limits_{\mathcal{O}}e_{j}^{2}\varphi \Delta \varphi d\xi +\frac{1}{2}%
\int\limits_{\mathcal{O}}e_{j}^{2}\left( 2\varphi \Delta \varphi +2\abs{
\nabla \varphi } ^{2}\right) d\xi \\
&&-\frac{1}{2}\int\limits_{\mathcal{O}}e_{j}^{2}\frac{\partial \varphi ^{2}%
}{\partial \nu }d\sigma +\int\limits_{\partial \mathcal{O}}\frac{%
\partial \left( e_{j}\varphi \right) }{\partial \nu }\left( e_{j}\varphi
\right) d\sigma +\int\limits_{{\Gamma_u}}\alpha \left\vert
e_{j}\varphi \right\vert ^{2}d\sigma \\
&=&\int\limits_{\mathcal{O}}\left( \lambda _{j}e_{j}^{2}\varphi ^{2}d\xi
+e_{j}^{2}\abs{\nabla \varphi} ^{2}\right) d\xi -\frac{1}{2}%
\int\limits_{\partial \mathcal{O}}e_{j}^{2}\frac{\partial \varphi ^{2}}{\partial
\nu }d\sigma \\
&&+\int\limits_{\partial \mathcal{O}}\frac{\partial \left( e_{j}\varphi
\right) }{\partial \nu }\left( e_{j}\varphi \right) d\sigma
+\int\limits_{\Gamma _{u}}\alpha \left\vert e_{j}\varphi \right\vert
^{2}d\sigma {=:}I+B_{1}+B_{2}+B_{3}.
\end{eqnarray*}
Let us write%
\begin{eqnarray*}
B_{2} &=&\int\limits_{\partial \mathcal{O}}\frac{\partial \left(
e_{j}\varphi \right) }{\partial \nu }\left( e_{j}\varphi \right) d\sigma
=\int\limits_{\partial \mathcal{O}}\frac{1}{2}\frac{\partial \left(
e_{j}^{2}\varphi ^{2}\right) }{\partial \nu }d\sigma
=\int\limits_{\partial \mathcal{O}}\frac{1}{2}\left( \frac{\partial
\left( e_{j}^{2}\right) }{\partial \nu }\varphi ^{2}+\frac{\partial \left(
\varphi ^{2}\right) }{\partial \nu }e_{j}^{2}\right) d\sigma .
\end{eqnarray*}
Since 
\begin{equation*}
\frac{1}{2}\frac{\partial \left( e_{j}^{2}\right) }{\partial \nu }=2e_{j}%
\frac{\partial e_{j}}{\partial \nu }=\left\{ 
\begin{array}{ll}
0, & \text{on }\Gamma _{s}, \\ 
-\alpha e_{j}^{2}, & \text{on }\Gamma _{u},%
\end{array}%
\right.
\end{equation*}%
it follows that%
\begin{equation*}
B_{2}=-\int\limits_{\partial \mathcal{O}}\alpha \left\vert e_{j}\varphi
\right\vert ^{2}d\sigma +\frac{1}{2}\int\limits_{\partial \mathcal{O}}\frac{%
\partial \left( \varphi ^{2}\right) }{\partial \nu }e_{j}^{2}d\sigma .
\end{equation*}
Replacing in the previous relation yields%
\begin{equation*}
\left\vert e_{j}\varphi \right\vert _{V}^{2}=\int\limits_{\mathcal{O}%
}\left( \lambda _{j}e_{j}^{2}\varphi ^{2}d\xi +e_{j}^{2}\abs{ \nabla
\varphi} ^{2}\right) d\xi.
\end{equation*}%
Then%
\begin{equation*}
\left\vert e_{j}\varphi \right\vert _{V}^{2}\leq C\lambda
_{j}^{\frac{d+1}{2}}\int\limits_{\mathcal{O}}\varphi ^{2}d\xi +C\lambda
_{j}^{\frac{d-1}{2}}\int\limits_{\mathcal{O}}\left\vert \nabla \varphi \right\vert
^{2}d\xi \leq C\pr{1+\lambda _{j}^{\frac{d+1}{2}}}\left\vert \varphi \right\vert _{V}^{2}.
\end{equation*}
As a consequence,
\begin{equation*}
\left\vert \varphi \right\vert _{V}^{2}=\left\langle e_{j}x,\varphi
\right\rangle _{2}\leq \left\vert x\right\vert _{V^{\prime }}\left\vert
e_j\varphi \right\vert _{V}\leq C\sqrt{\pr{1+\lambda _{j}^{\frac{d+1}{2}}}}\left\vert \varphi \right\vert _{V}\abs{x}_{V'}.
\end{equation*}%
Going back to (\ref{star}), one gets, by combining the last two inequalities, the remaining assertion of our proposition.
\end{proof}

\subsection{Proofof Lemma \ref{lema2}}
\begin{proof}[Proof of Lemma \ref{lema2}]
We write the weak form of the difference%
\begin{equation*}
J_{\lambda }^{\mu, \varepsilon }\left( y\right) -J_{\lambda }^{\mu,\varepsilon
}\left( \overline{y}\right) +\varepsilon A_{\lambda }^{\mu }\left(
J_{\lambda }^{\mu, \varepsilon }\left( y\right) \right) -\varepsilon A_{\lambda
}^{\mu }\left( J_{\lambda }^{\mu, \varepsilon }\left( \overline{y}\right) \right)
=y-\overline{y},\quad in\text{ }V^{\prime },
\end{equation*}%
as%
\begin{equation*}
_{V^{\prime }}\left\langle J_{\lambda }^{\mu, \varepsilon }\left( y\right)
-J_{\lambda }^{\mu, \varepsilon }\left( \overline{y}\right) ,\zeta \right\rangle
_{V}+\varepsilon _{V^{\prime }}\left\langle A_{\lambda }^{\mu }\left(
J_{\lambda }^{\mu, \varepsilon }\left( y\right) \right) -A_{\lambda }^{\mu
}\left( J_{\lambda }^{\mu, \varepsilon }\left( \overline{y}\right) \right) ,\zeta
\right\rangle _{V}=_{V^{\prime }}\left\langle y-\overline{y},,\zeta
\right\rangle _{V},
\end{equation*}%
for $\zeta \in V$ and we take the particular element $\zeta =\widetilde{\Psi }_{\lambda
}\left( J_{\lambda }^{\mu, \varepsilon }\left( y\right) \right) -\widetilde{\Psi }%
_{\lambda }\left( J_{\lambda }^{\mu, \varepsilon }\left( \overline{y}\right)
\right) \in V$ to get

\begin{eqnarray*}
&&_{V^{\prime }}\left\langle J_{\lambda }^{\mu, \varepsilon }\left( y\right)
-J_{\lambda }^{\mu, \varepsilon }\left( \overline{y}\right) ,\widetilde{\Psi }%
_{\lambda }\left( J_{\lambda }^{\mu, \varepsilon }\left( y\right) \right) -%
\widetilde{\Psi }_{\lambda }\left( J_{\lambda }^{\mu, \varepsilon }\left( 
\overline{y}\right) \right) \right\rangle _{V} \\
&&+\varepsilon \mu \int_{\mathcal{O}}\left( J_{\lambda }^{\mu,\varepsilon
}\left( y\right) -J_{\lambda }^{\mu, \varepsilon }\left( \overline{y}\right)
\right) \left( \widetilde{\Psi }_{\lambda }\left( J_{\lambda }^{\mu,\varepsilon
}\left( y\right) \right) -\widetilde{\Psi }_{\lambda }\left( J_{\lambda
}^{\mu, \varepsilon }\left( \overline{y}\right) \right) \right) d\xi \\
&&+\varepsilon \int_{\mathcal{O}}\left| \nabla \widetilde{\Psi }_{\lambda
}\left( J_{\lambda }^{\mu, \varepsilon }\left( y\right) \right) -\nabla 
\widetilde{\Psi }_{\lambda }\left( J_{\lambda }^{\mu, \varepsilon }\left( 
\overline{y}\right) \right) \right| ^{2}d\xi \\
&&-\varepsilon \int_{\mathcal{O}}K\left( J_{\lambda }^{\mu, \varepsilon }\left(
y\right) -J_{\lambda }^{\mu, \varepsilon }\left( \overline{y}\right) \right)
\left( \nabla \widetilde{\Psi }_{\lambda }\left( J_{\lambda }^{\mu,\varepsilon
}\left( y\right) \right) -\nabla \widetilde{\Psi }_{\lambda }\left(
J_{\lambda }^{\mu, \varepsilon }\left( \overline{y}\right) \right) \right)\cdot\ i_3 d\xi \\
&&+\varepsilon \int_{\Gamma _{u}}\alpha Tr \left( \widetilde{\Psi }_{\lambda
}\left( J_{\lambda }^{\mu, \varepsilon }\left( y\right) \right) -\widetilde{\Psi }%
_{\lambda }\left( J_{\lambda }^{\mu, \varepsilon }\left( \overline{y}\right)
\right) \right) ^{2}d\sigma=\int_{\mathcal{O}}\left( y-\overline{y}\right) \left( \widetilde{\Psi }%
_{\lambda }\left( J_{\lambda }^{\mu, \varepsilon }\left( y\right) \right) -%
\widetilde{\Psi }_{\lambda }\left( J_{\lambda }^{\mu, \varepsilon }\left( 
\overline{y}\right) \right) \right) d\xi .
\end{eqnarray*}

From the $\lambda$-strong monotonicity of $\widetilde{\Psi }_{\lambda }$, we have that%
\begin{eqnarray*}
&&\left( \lambda +\varepsilon \mu \lambda \right) \int_{\mathcal{O}}\left(
J_{\lambda }^{\mu, \varepsilon }\left( y\right) -J_{\lambda }^{\mu,\varepsilon
}\left( \overline{y}\right) \right) ^{2}d\xi +\varepsilon \int_{\mathcal{O}}\left| \nabla \widetilde{\Psi }_{\lambda
}\left( J_{\lambda }^{\mu, \varepsilon }\left( y\right) \right) -\nabla 
\widetilde{\Psi }_{\lambda }\left( J_{\lambda }^{\mu, \varepsilon }\left( 
\overline{y}\right) \right) \right| ^{2}d\xi \\
&&-\varepsilon K\int_{\mathcal{O}}\left( J_{\lambda }^{\mu, \varepsilon }\left(
y\right) -J_{\lambda }^{\mu, \varepsilon }\left( \overline{y}\right) \right)
\left( \nabla \widetilde{\Psi }_{\lambda }\left( J_{\lambda }^{\mu,\varepsilon
}\left( y\right) \right) -\nabla \widetilde{\Psi }_{\lambda }\left(
J_{\lambda }^{\mu, \varepsilon }\left( \overline{y}\right) \right) \right)\cdot\ i_3 d\xi \\
&&+\varepsilon \int_{\Gamma _{u}}\alpha Tr \left( \widetilde{\Psi }_{\lambda
}\left( J_{\lambda }^{\mu, \varepsilon }\left( y\right) \right) -\widetilde{\Psi }%
_{\lambda }\left( J_{\lambda }^{\mu, \varepsilon }\left( \overline{y}\right)
\right) \right) ^{2}d\sigma \leq\int_{\mathcal{O}}\left( y-\overline{y}\right) \left( \widetilde{\Psi 
}_{\lambda }\left( J_{\lambda }^{\mu, \varepsilon }\left( y\right) \right) -%
\widetilde{\Psi }_{\lambda }\left( J_{\lambda }^{\mu, \varepsilon }\left( 
\overline{y}\right) \right) \right) d\xi .
\end{eqnarray*}
\noindent Elementary computations yield%
\begin{eqnarray*}
&&\left( \lambda +\varepsilon \mu \lambda -\varepsilon K^{2}\right) \int_{%
\mathcal{O}}\left( J_{\lambda }^{\mu, \varepsilon }\left( y\right) -J_{\lambda
}^{\mu, \varepsilon }\left( \overline{y}\right) \right) ^{2}d\xi +\varepsilon \int_{\Gamma _{u}}\alpha Tr \left( \widetilde{\Psi }_{\lambda
}\left( J_{\lambda }^{\mu, \varepsilon }\left( y\right) \right) -\widetilde{\Psi }%
_{\lambda }\left( J_{\lambda }^{\mu, \varepsilon }\left( \overline{y}\right)
\right) \right) ^{2}d\sigma \\
&\leq &\int_{\mathcal{O}}\left( y-\overline{y}\right) \left( \widetilde{\Psi 
}_{\lambda }\left( J_{\lambda }^{\mu, \varepsilon }\left( y\right) \right) -%
\widetilde{\Psi }_{\lambda }\left( J_{\lambda }^{\mu, \varepsilon }\left( 
\overline{y}\right) \right) \right) d\xi\\&\leq&\frac{1}{2\varepsilon }\int_{\mathcal{O}}\left( y-\overline{y}\right)
^{2}d\xi +\frac{\varepsilon}{2\lambda^2 }\int_{%
\mathcal{O}}\left( J_{\lambda }^{\mu, \varepsilon }\left( y\right) -J_{\lambda
}^{\mu,\varepsilon }\left( \overline{y}\right) \right) ^{2}d\xi .
\end{eqnarray*}%
and, therefore,
\begin{equation*}
\left( \lambda +\varepsilon \mu \lambda -\varepsilon K^{2}-\frac{\varepsilon}{2\lambda^2 }\right) \int_{\mathcal{O}}\left(
J_{\lambda }^{\mu, \varepsilon }\left( y\right) -J_{\lambda }^{\mu,\varepsilon
}\left( \overline{y}\right) \right) ^{2}d\xi \leq \frac{1}{2}\int_{\mathcal{O%
}}\left( y-\overline{y}\right) ^{2}d\xi .
\end{equation*}%
For $\mu$ sufficiently large e.g., larger than $\frac{K^2+1}{\lambda^3}$, we have the Lipschitz continuity of $%
J_{\lambda }^{\mu, \varepsilon }$ for every $0<\lambda,\varepsilon<1$. Note that, alternatively, the conclusion can be obtained with a fixed $\mu$ by requiring $\varepsilon$ to be small enough.
\end{proof}

\section*{Acknowledgement}
I.C. was supported by the DEFHY3GEO project.\\
    
\noindent D.G. and J.L. acknowledge financial support from the National Key R and D Program of China (No. 2018YFA0703900), the NSF of Shandong Province (No. ZR202306020015), the NSF of P.R.China (Nos. 12031009, 11871037).\\

\noindent I.C and A.T would like to thank the Shandong University where part of this work was developed.

\end{document}